\documentclass[11pt,a4paper,reqno]{amsart}

\usepackage{amsmath}
\usepackage{amsthm}
\usepackage{amsfonts, amssymb}
\usepackage[all]{xy}
\usepackage{url}
\usepackage{algorithm}
\usepackage[end]{algpseudocode}
\usepackage{framed}
\usepackage{faktor}
\usepackage{tabularx}

\usepackage{latexsym}

\usepackage{graphicx}
\usepackage{graphics}
\usepackage[scriptsize,bf]{caption}
\usepackage{float}
\usepackage{enumerate}
\usepackage{verbatim}
\usepackage{color}
\usepackage{subfig}
\usepackage{tikz}

\theoremstyle{plain}
\newtheorem{lema}{Lemma}
\newtheorem{prop}[lema]{Proposition}
\newtheorem{teo}[lema]{Theorem}

\newtheorem{coro}[lema]{Corollary}

\theoremstyle{definition}
\newtheorem{defi}[lema]{Definition}

\newtheorem{obs}[lema]{Remark}

\makeatletter
\@namedef{subjclassname@2020}{%
  \textup{2020} Mathematics Subject Classification}
  
\newcommand{\multiline}[1]{%
  \begin{tabularx}{\dimexpr\linewidth-\ALG@thistlm}[t]{@{}X@{}}
    #1
  \end{tabularx}
}  
  
\makeatother

\newcommand{\m}{\mathrm{m}}

\begin{document}

\title[Weak harmonic labeling of graphs and multigraphs]{Weak harmonic labeling of graphs and multigraphs}

\author[P. L. Bonucci]{Pablo L. Bonucci}
\author[N. A. Capitelli]{Nicol\'as A. Capitelli\\
\textit{\scriptsize
U\MakeLowercase{niversidad} N\MakeLowercase{acional de} L\MakeLowercase{uj\'an}, D\MakeLowercase{epartamento de} C\MakeLowercase{iencias} B\MakeLowercase{\'asicas}, A\MakeLowercase{rgentina.}}}

\thanks{\textit{Corresponding address:} {\color{blue}ncapitelli@unlu.edu.ar}}

\subjclass[2020]{05C78, 05C63, 05C75, 05C85}

\keywords{Graph labeling, Harmonic functions on graphs, Harmonic labeling}

\thanks{\textit{This research was partially supported by the Department of Basic Sciences, UNLu (Disposici\'on CDD-CB 148/18) and the University of Luj\'an (Resoluci\'on PHCS 66/18)}}

\begin{abstract} In this article we introduce the notion of \emph{weak harmonic labeling} of a graph, a generalization of the concept of harmonic labeling defined recently by Benjamini et al. that allows extension to finite graphs and graphs with leaves. We present various families of examples and provide several constructions that extend a given weak harmonic labeling to larger graphs. In particular, we use finite weak models to produce new examples of (strong) harmonic labelings. As a main result, we provide a characterization of weakly labeled graphs in terms of \emph{harmonic subsets} of $\mathbb{Z}$ and use it to compute every such graphs of up to ten vertices. In particular, we characterize harmonically labeled graphs as defined by Benjamini et al. We further extend the definitions and main results to the case of multigraphs and total labelings.\end{abstract}

\maketitle

\section{Introduction}

The notion of harmonic labeling of an infinite (simple) graph was introduced recently by Benjamini, Cyr, Procaccia and Tessler in \cite{BCPT}. If $G=(V,E)$ is an infinite graph of bounded degree then an \textit{harmonic labeling} of $G$ is a bijective function $\ell:V\rightarrow \mathbb{Z}$ such that\begin{equation}\label{Def:IntroHL}\ell(v)=\frac{1}{\deg(v)}\sum_{\{v,w\}\in E}\ell(w)\end{equation} for every $v\in V$. In \cite{BCPT} the authors provide some examples of harmonic labelings and prove the existence of such labelings for regular trees and the lattices $\mathbb{Z}^d$ and the non-existence  for cylinders $G\times \mathbb{Z}$ for non-trivial $G$. Graph labeling is a widely developed topic and has a broad range of applications (see, e.g., \cite{BG1,BG2,Gallian}).



Harmonically labellable graphs seem to have a rather restrictive configuration. Par\-ti\-cu\-lar\-ly, these graphs do not have leaves (pendant vertices) since there are no one to one functions verifying harmonicity on such vertices. Actually, this turns out to be the main obstacle for a generalization of this concept to the context of finite graphs, which is a natural extension taking into account the fruitful link between harmonic functions and geometric properties of finite graphs (see e.g. \cite[\S 4]{Lov}). Furthermore, finite examples might be useful as local models to produce new harmonically labeled (infinite) graphs.

In this paper we propose a two-way generalization of the notion of harmonic labeling, introducing the concept of \textit{weak harmonic labeling}. On one hand, we require satisfying equation \eqref{Def:IntroHL} only for $v\in V\setminus S$, where $S$ is the set of leaves of $G$. On the other hand, we let the function $\ell$ be a bijection with an integer interval $I$ (finite or infinite). These conditions permit a straighforward extension of harmonic labelings to the finite setting. This results in a more general structure which provides a far wider theory, which was one the of the ambitions in \cite{BCPT}.


We present several examples of weak harmonic labelings and show the non-existence of this type of labelings for various families of (finite and infinite) graphs. We further introduce constructions to obtain new examples from given ones. In particular, we define the notion of inner cylinder and a way to extend any weakly labeled finite graph into an infinite one. We use weak finite models to construct new families of harmonic labelings. In particular, we exhibit a non-numerable collection of harmonically labeleled graphs, which additionally contains an infinite number of examples spanned by finite sets of vertices, thus answering a question raised in \cite{BCPT} (see Remark \ref{Remark:OpenProblems}).


The main result of this article is the characterization of weakly  labeled graphs in terms of certain families of collections of finite subsets of $\mathbb{Z}$ called \emph{harmonic subsets}. Since the statement of this result without many preliminary conventions would be too lengthy, the reader is invited to turn to Lemma \ref{Lemma: 4 Properties} and Theorem \ref{Theorem:MAIN} for a first impression. This characterization provides a way to compute all weakly labeled graphs, thing which we do for graphs of up to ten vertices (see Appendix). In particular, we obtain a characterization of harmonically labeled graphs, as defined in \cite{BCPT}, in terms of the aforementioned harmonic subsets (Theorem \ref{Theorem:MAINMAIN}).

All the definitions and results of weak harmonic labelings can be extended to the case of multigraphs (or total labelings) in a straightforward way. We prove the version for multigraphs of Theorem \ref{Theorem:MAIN} and exhibit an algorithm that produces a total weak harmonic labeling from a given admissible labeling (see Algorithm \ref{algovalues}).

The paper is organized as follows. In Section 2 we introduce the concept of harmonic labeling and exhibit several examples of (families) of weakly labeled (finite and infinite) graphs.  In Section 3 we  present two constructions to obtain a new labelings from a given one and we use finite models of weakly  labeled graphs to construct new families of harmonically labeled graphs. In Section 4 we prove the characterization of weakly  labeled graphs (and, in particular, of harmonic labelings) in terms of families of collections of \emph{harmonic subsets of $\mathbb{Z}$}. In Section 5 we extend the definitions and main results of the theory to the case of multigraphs and total labelings. In the Appendix we have included the list of all possible weakly labeled graphs up to ten vertices.



\section{Weak Harmonic Labelings of Simple Graphs}\label{Section:HL}


\textit{All graphs considered are connected, have bounded degree and at least three vertices}. For a simple graph $G$ we write $V_G$ for it set of vertices and $E_G$ for its set of edges. We put $v \sim w$ if $v$ and $w$ are adjacent and we let $N_G(v)=\{v\}\cup\{w\,:\, w\sim v\}\subset V_G$ denote the closed neighborhood of $v$.  Throughout, $S_G$ will denote the set of leaves (pendant vertices) of $G$ and $I$ will denote a generic integer interval (a set of consecutive integers).



\begin{obs}\label{Obs:dosverticesunonohoja} Note that, for any $G$, $v\sim w$ implies $\{v,w\}\cap (V_G\setminus S_G)\neq\emptyset$.\end{obs}



\begin{defi}\label{Def:Harmonic} A \textit{weak harmonic labeling} of a graph $G$ (simply \emph{weak labeling} in this context) is a bijective function $\ell: V_G\rightarrow I$ such that \begin{equation}\label{Eq:HarmonicityCondition} \ell(v)=\frac{1}{deg(v)}\sum_{w\sim v}\ell(w)\hspace{0.2in}\forall v \in V_G\setminus S_G.
\end{equation} When we want to explicitate the interval of the labeling, we shall say \emph{weak harmonic labeling onto $I$}.
\end{defi}
As mentioned earlier, the relativeness to $V_G\setminus S_G$ of the harmonicity property is natural as there cannot be one to one functions with harmonic leaves. Harmonic labelings are particular cases of weak harmonic labelings since harmonically labellable infinite graphs have no leaves. More precisely, a weak harmonic labeling onto $I$ is an harmonic labeling if and only if $I=\mathbb{Z}$ and $S_G=\emptyset$.
%
%
%

\begin{obs} Since a function $\ell$ satisfies equation \eqref{Eq:HarmonicityCondition} if and only if $\pm \ell+k$ satifies it for any $k\in\mathbb{Z}$, we shall not distinguish between labelings obtained from translations or invertions. Thus, we make the convention that in the case $I\neq \mathbb{Z}$ we shall normalize all labelings to the intervals $[0,|V_G|-1]=\{k\in\mathbb{Z}\,:\, 0\leq k\leq |V_G|-1\}$ or $[0,\infty]=\{k\in\mathbb{Z}\,:\, k\geq 0\}$.\end{obs}

The simplest examples of weakly labeled finite graphs are the paths $P_n$ and the stars $K_{1,n}$ for even $n$ (Figure \ref{Figure:Combination}). Paths can be extended either to $\infty$ or to both $-\infty$ and $\infty$ to obtain a weak harmonic labeling onto $[0,\infty]$ or $\mathbb{Z}$ respectively. In the latter, we obtain the trivial harmonically labeled graph $\mathbb{Z}$. We invite the reader to check the Appendix for a numerous (concrete) examples of weakly labeled finite graphs, where additionally it can be verified that a given graph can admit more than one weak harmonic labeling.

\begin{figure}
\begin{center}
\begin{tikzpicture}[y=1.5cm, x=2.5cm, scale=0.6]
	\node (0)at(-7,0){};
	\filldraw[fill=green!80,draw=black!80] (0) circle (3pt)
	node[below=2pt] {\small $0$};
	\node (1)at(-6,0){};
	\filldraw[fill=red!80,draw=black!80] (1) circle (3pt)
	node[below=2pt] {\small $1$};
	\node (2)at(-5,0){};
	\filldraw[fill=red!80,draw=black!80] (2) circle (3pt)
	node[below=2pt] {\small $2$};
	\node (3)at(-4,0){};
	\filldraw[fill=red!80,draw=black!80] (3) circle (3pt)
	node[below=2pt] {\small $3$};
	\node (4)at(-3,0){};
	\filldraw[fill=red!80,draw=black!80] (4) circle (3pt)
	node[below=2pt] {\small $n-2$};
	\node (5)at(-2,0){};
	\filldraw[fill=green!80,draw=black!80] (5) circle (3pt)
	node[below=2pt] {\small $n-1$};
	\draw (1)--(0);
	\draw (1)--(2);
	\draw (3)-- (2);
	\draw (3)[dashed]--(4);
	\draw (5)--(4);

	\hspace{10mm}
	
\small	
	
\def \n {18}
\def \N {8}
\def \radius {3cm}
\def \rd {1mm}
\def \rer {10mm}

\def \margin {8} 

\node at (360:0mm) (ustar){};	
\filldraw[fill=red!80,draw=black!80] (ustar) circle (3pt)
node[below=2pt]{$\frac{n}{2}$};

\foreach \i [count=\ni from 0] in {\tiny{n},0,1,2,3}{
\node at ({120-\ni*25}:\radius) (u\ni){};
\filldraw[fill=green!80,draw=black!80]  (u\ni) circle (3pt)
node[above=2pt]{$\i$};

\node at ({115-\ni*18}:\radius/2) {};
\draw (ustar)--(u\ni);
}

\foreach \i in {1,3,...,11}{
  \node[circle] at ({-\i*20}:\radius) (aux) {\phantom{$w_{5}$}};
  \draw[dotted, shorten >=3mm, shorten <=3mm] (ustar)--(aux);
}

\draw[dotted,red] (18:\radius/2) arc[start angle=18, end angle=-228, radius=\radius/2];

\end{tikzpicture}
\end{center}
\caption{Weak harmonic labeling on $P_n$ and $K_{1,n}$.}
\label{Figure:Combination}
\end{figure}
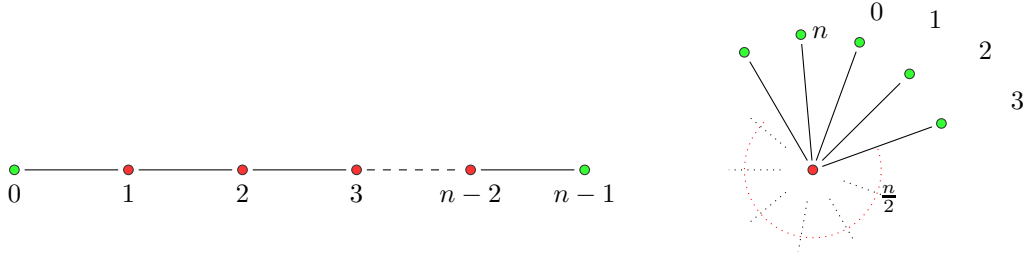

Note that the minimum and maximum values of a weak harmonic labeling over a finite $G$ must take place on leaves, so any finite graph with less than two leaves does not admit weak harmonic labelings. This is the analogue result that nonconstant harmonic funtions have at least two poles (see e.g. \cite[\S 4]{Lov}). In particular, cycles, complete graphs $K_n$ with $n\geq 3$, complete bipartite graphs $K_{n,m}$ with $n,m\geq 2$ and cylinders $G\times P_n$ for $n\geq 2$ and any $G$ do not admit a weak harmonic labeling. It is not hard to characterize finite graphs with maximum and minimum number of leaves which admit this type of labeling.

\begin{lema} Let $G$ be an $n$-vertex graph which admits a weak harmonic labeling.\begin{enumerate}
\item $G$ has two leaves if and only if $G=P_{n}$.
\item $G$ has $n-1$ leaves if and only if $n$ is even and $G=K_{1,n}$.\end{enumerate} \end{lema}

\begin{proof} We prove the direct of ($1$), which is the only non-trivial implication. Let $\ell:V_G\rightarrow I$ be a weak harmonic labeling of $G$ and denote $v_i$ be the vertex labeled $i$. We may assume $n\geq 4$. By the previous remarks, $v_0$ and $v_n$ are the leaves of $G$. Since the vertex $v_1\notin S_G$  then $v_1\sim v_0$. Now
$$\deg(v_1)=\sum_{w \sim v_1}\ell(w)\geq \sum_{\substack{w \sim v_1 \\ w\neq v_0}}2=2(\deg(v_1)-1),$$ from where $\deg(v_1)= 2$. Therefore, $N_{G}(v_1)=\{v_0,v_2\}$. The same argument shows that $N_{G}(v_{n-1})=\{v_{n-2},v_n\}$. Assume inductively that $N_{G}(v_i)=\{v_{i-1},v_{i+1}\}$ for $0<i<k<n-1$. Then
$$\deg(v_k)k=\sum_{w \sim v_k}\ell(w)\geq (k-1)+\sum_{\substack{w \sim v_k \\ w\neq v_{k-1}}}(k+1)=k-1+(k+1)(\deg(v_k)-1),$$
and $\deg(v_k)\leq 2$. This proves that $N_{G}(v_k)=\{v_{k-1},v_{k+1}\}$ and hence $G=P_n$.\end{proof}

More general families of weakly labeled finite graphs are shown in Figure \ref{Figure:CkGeneralized}. Note that $P_n$ and $K_{1,n}$ ($n$ even) are extremal cases of the collection pictured in Figure \ref{Figure:CkGeneralized} (top). The non-acyclic family in Figure \ref{Figure:CkGeneralized} (bottom), which can be inferred from the examples in the Appendix, can be trivially extended to labelings onto $[0,\infty]$ and $\mathbb{Z}$. In the latter, we obtain again an harmonic labeling. Furthermore, another such labeling for this graph can be produced by adding the edges $\{\{2k-1,2k+1\}\,|\, k\in \mathbb{Z}\}$. These two examples are different from all those present in \cite{BCPT}, which evidences how new examples of harmonic labelings can be deduced from finite weakly labeled ones. We shall present more examples obtained in this fashion in the next section.

\begin{obs} Recall that the \emph{Laplacian} of a finite graph $G$ is the operator $L_G= D-A\in\mathbb{Z}^{n\times n}$ where $A$ is the adjacency matrix of $G$ and $D$ is the diagonal degree matrix. If we let $\tilde{L}_G$ denote the operator obtained from $L_G$ by removing the rows corresponding to leaves (the \emph{reduced Laplacian} of $G$) then $G$ admits a weak harmonic labeling if and only if there exists a permutation $\sigma\in S_n$ such that $\sigma(0,\ldots,n-1)\in ker(\tilde{L}_G)$.\end{obs}




\begin{figure}

\begin{tikzpicture}[x=1.3cm, scale=0.5]
\small 
\def \n {18}
\def \N {8}
\def \radius {3cm}
\def \rd {1mm}
\def \rer {10mm}	
	\node (3)at(0,0){};
	\filldraw[fill=red!80,draw=black!80](3)circle (3pt);
	\draw (-3.5,1.5) node {$\dfrac{(n-1)m}{2}$};
	
	\node (1)at(-8,0){};
	\filldraw[fill=green!80,draw=black!80](1)circle (3pt);
	\draw (-8,-1) node {$km$};
	
	\node (2)at(-5,0){};
	\filldraw[fill=red!80,draw=black!80](2)circle (3pt);
	\draw (-5,-1) node {$(k+1)m$};

	\node (4)at(5,0){};
	\filldraw[fill=red!80,draw=black!80](4)circle (3pt);
	\draw (5,1) node {$(n-k-2)m$};
	
	
	\node (5)at(8,0){};
	\filldraw[fill=green!80,draw=black!80](5)circle (3pt);
	\draw (8,-1) node {$(n-k-1)m$};
	
	
	\node (10)at(-2,3){};
	\filldraw[fill=green!80,draw=black!80](10)circle (3pt);
     \draw (0,4.1) node {$s\not\equiv 0 (m)$};
     \draw (0,3.5) node {$\overbrace{\hspace{1in}}$};
     \draw (0,4.9) node {$s\in [0,(n-1)m]$};
     \node (11)at(2,3){};
	\filldraw[fill=green!80,draw=black!80](11)circle (3pt);

	
	\node (12)at(-2,-3){};
	\filldraw[fill=green!80,draw=black!80](12)circle (3pt);
	\node (13)at(2,-3){};
	\filldraw[fill=green!80,draw=black!80](13)circle (3pt);
	
     \draw (0,-4.1) node {$s=tm$};
     \draw (0,-3.5) node {$\underbrace{\hspace{1in}}$};
     \draw (0,-4.9) node {$t\in [0,k-1]\cup [n-k,n-1]$};	
	
	%
	
	\draw[-,dotted] (-1.2,-2.5) to[bend right]  (1.2,-2.5);
	\draw[-,dotted] (-1.2,2.5) to[bend left]  (1.2,2.5);	
	
	\draw[->] (3) to[bend left] (-2.5,1);
	\draw  (1) to (2);
	\draw[dashed]  (2) -- (3);
	\draw[dashed]  (3) -- (4);
	\draw  (4) to (5);
		\draw  (3) to (10);
			\draw  (3) to (11);
					\draw  (3) to (12);
							\draw  (3) to (13);
	\end{tikzpicture}	

\vspace{0.3in}	
	
	\begin{tikzpicture}[scale=0.7]
	\small 
	
	\node(0)at(0,0){};
	\filldraw[fill=green!80,draw=black!80] (0) circle (2pt)
	node[above=2pt] {$0$}; 
	\node (2)at(2 ,0){};	
	\filldraw[fill=red!80,draw=black!80] (2) circle (2pt)
	node[above=2pt] {$2$};
	\node (4)at(4,0){};	
	\filldraw[fill=red!80,draw=black!80] (4) circle (2pt)
	node[above=2pt] {$4$};
	\node (6)at(6,0){};	
	\filldraw[fill=red!80,draw=black!80] (6) circle (2pt)
	node[above=2pt] {$6$};
	\node (8)at( 8,0){};	
	\filldraw[fill=red!80,draw=black!80] (8) circle (2pt)
	node[above=2pt] {};
	\node (10)at(10,0){};	
	\filldraw[fill=red!80,draw=black!80] (10) circle (2pt)
	node[above=2pt] {};
	\node (12)at(12,0){};	
	\filldraw[fill=red!80,draw=black!80] (12) circle (2pt)
	node[above=2pt] {$2n-4$};	
	\node (14)at(14,0){};	
	\filldraw[fill=red!80,draw=black!80] (14) circle (2pt)
	node[above=2pt] {$2n-2$};
	\node (16)at(16,0){};	
	\filldraw[fill=green!80,draw=black!80] (16) circle (2pt)
	node[above=2pt] {$2n$};
	
	\node (1)at(1,-2){};	
	\filldraw[fill=green!80,draw=black!80] (1) circle (2pt)
	node[below=2pt] {$1$};	
	\node (3)at(3,-2){};	
	\filldraw[fill=red!80,draw=black!80] (3) circle (2pt)
	node[below=2pt] {$3$};
	\node (5)at(5,-2){};	
	\filldraw[fill=red!80,draw=black!80] (5) circle (2pt)
	node[below=2pt] {$5$};
	\node (7)at(7,-2){};	
	\filldraw[fill=red!80,draw=black!80] (7) circle (2pt)
	node[below=2pt] {};
	
	\node (11)at(11,-2){};	
	\filldraw[fill=red!80,draw=black!80] (11) circle (2pt)
	node[below=2pt] {};
	\node (13)at(13,-2){};	
	\filldraw[fill=red!80,draw=black!80] (13) circle (2pt)
	node[below=2pt] {$2n-3$};
	\node (15)at(15,-2){};	
	\filldraw[fill=green!80,draw=black!80] (15) circle (2pt)
	node[below=2pt] {\small $2n-1$};

	\draw (0) to (2);
	\draw (1) to (2);
	\draw (2) to (3);
	\draw (2) to (4);
	\draw (3) to (4);
	\draw (4) to (5);
	\draw (4) to (6);
	\draw (5) to (6);
	
	\draw (6) [dashed] -- (8);
	\draw (6) [dashed] -- (7);
	\draw (7) [dashed] -- (8);
	\draw (8) [dashed] -- (10);
	\draw (10) [dashed] -- (12);
	\draw (10) [dashed] -- (11);
	\draw (11) [dashed] -- (12);
	
	\draw (12) to (13);
	\draw (12) to (14);
	\draw (13) to (14);	
	\draw (14) to (15);
	\draw (14) to (16);

	\end{tikzpicture}
\normalsize


\caption{\textbf{Top:} A collection of weakly labeled graphs for $m,n,k\in\mathbb{Z}_{\geq 0}$, $m \geq 1$, $n$ odd and $0\leq k\leq \frac{(n-1)m}{2}$. The graphs $P_n$ and $K_{1,n}$ ($n$ even) are extremal cases of this family for $m=1$. \textbf{Bottom:} A family of weakly labeled finite graphs than can additionally be extended to weakly labeled graphs onto $[0,\infty]$ and $\mathbb{Z}$.}
\label{Figure:CkGeneralized}
\end{figure}
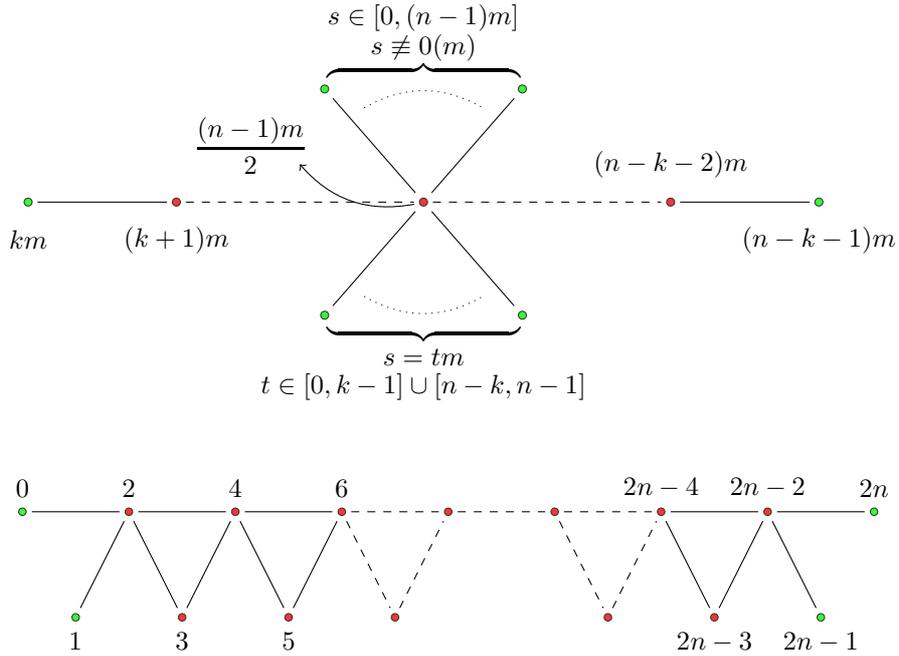

\section{Harmonic labelings from finite weak models}



More complex weakly labeled (finite and infinite) graphs can be built up from simpler finite examples. Some of these graphs can be inferred from the structure of the finite model and some can be constructed by performing unions and considering cylinders on them. In many cases, we shall obtain (new) harmonically labeled graphs.


\subsection*{Coalescence and Inner Cylinders} We first show two constructions that produce new weakly labeled (finite and infinite) graphs from a finite weak model. Particularly, these constructions provide a way to produce infinitely many weak harmonic labelings onto $[0,\infty]$ and $\mathbb{Z}$.

For simple graphs $G,H$ and $v\in V_G$ and $w\in V_H$ we let $G\cdot_v^w H$ denote the graph obtained from $G\cup H$ by identifying the vertex $v$ with the vertex $w$ (this is sometimes referred by some authors as the \textit{coalescence} between $G$ and $H$ at vertices $v$ and $w$).

\begin{lema}\label{Lemma:identificacionhojas} Let $\ell_G:V_G\rightarrow [0,n-1]$ and $\ell_H:V_H\rightarrow I$, $I=[0,m-1]$ or $[0,\infty]$ be weak harmonic labelings on graphs $G$ and $H$ respectively. Let $v_i\in V_G$ be the vertex labeled $i$ in $G$ ($0\leq i \leq n-1$) and $w_j\in V_H$ be the vertex labeled $j$ in $H$ ($0\leq j \leq m-1$). If the sole vertex $v$ adjacent to $v_{n-1}$ in $G$ and the sole vertex $w$ adjacent to $w_0$ in $H$ satisfy $\ell_G(v)+\ell_H(w)=n-1$ then there exists a weak harmonic labeling of $G\cdot_{v_{n-1}}^{w_0} H$.\end{lema}
 \begin{proof} The desired weak harmonic labeling $\ell$ over $G\cdot_{v_{n-1}}^{w_0} H$ is given $$\ell(u)=\begin{cases}\ell_G(u)&u\in G\\ \ell_H(u)+n-1&u\in H.\end{cases}$$\end{proof}


The construction of Lemma \ref{Lemma:identificacionhojas} can be iterated to produce infinitely many new examples (both finite and infinite). Furthermore, \emph{any} weakly labeled graph can be extended to a new (finite or infinite) weakly labeled graph since the family of bipartite complete graphs $\{K_{1,n}\,:\,n\text{ even}\}$ has a member of average $k$ for each $k\in\mathbb{N}$. Figure \ref{Figure:Union} shows a particular example of this situation.


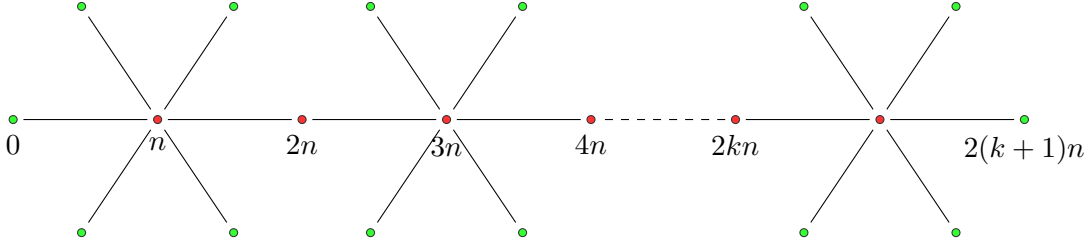
\begin{figure}
\vspace{5mm}
\begin{center}
\begin{tikzpicture}[scale=0.5]
	\node (3)at (0,0){};
	\filldraw[fill=red!80,draw=black!80](3) circle (3pt) 
	node[below=3pt] {$n$};
	
	\node (0)at (-3.8,0){};
	\filldraw[fill=green!80,draw=black!80](0) circle (3pt) 
	node[below=2pt] {$0$};

	\node (5)at (3.8,0){};
	\filldraw[fill=red!80,draw=black!80](5) circle (3pt) 
	node[below=2pt] {$2n$};

	\node (6)at(2,-3){};
	\filldraw[fill=green!80,draw=black!80](6)circle (3pt);
	\node (1)at(-2,-3){};
	\filldraw[fill=green!80,draw=black!80](1)circle (3pt);
	\node (2)at(-2,3){};
	\filldraw[fill=green!80,draw=black!80](2)circle (3pt);
	\node (4)at(2,3){};
	\filldraw[fill=green!80,draw=black!80](4)circle (3pt);
	
	\node (7)at (7.6,0){};
	\filldraw[fill=red!80,draw=black!80](7) circle (3pt) 
	node[below=3pt] {$3n$};
	
	\node (8)at (11.4,0){};
	\filldraw[fill=red!80,draw=black!80](8) circle (3pt) 
	node[below=2pt] {$4n$};	
	\node (9)at(9.6,-3){};
	\filldraw[fill=green!80,draw=black!80](9)circle (3pt);
	\node (10)at(5.6,-3){};
	\filldraw[fill=green!80,draw=black!80](10)circle (3pt);		
	\node (11)at(5.6,3){};
	\filldraw[fill=green!80,draw=black!80](11)circle (3pt);
	\node (12)at(9.6,3){};
	\filldraw[fill=green!80,draw=black!80](12)circle (3pt);
	
	\node (19)at (15.2,0){};
	\filldraw[fill=red!80,draw=black!80](19) circle (3pt) 
	node[below=2pt] {$2kn$};
	
	\node (13) at (19,0){};
	\filldraw[fill=red!80,draw=black!80] (13)circle (3pt);

	\node(14)at(22.8,0){};
	\filldraw[fill=green!80,draw=black!80] (14)circle (3pt) 
	node[below=2pt] {$2(k+1)n$};
	\node (15)at(21,-3){};
	\filldraw[fill=green!80,draw=black!80](15)circle (3pt);
	\node (16)at(17,-3){};
	\filldraw[fill=green!80,draw=black!80](16)circle (3pt);
	\node (17)at(17,3){};
	\filldraw[fill=green!80,draw=black!80](17)circle (3pt);
	\node (18)at(21,3){};
	\filldraw[fill=green!80,draw=black!80](18)circle (3pt);

	\draw  (13) to (14);
	\draw  (13) to (15);
	\draw  (13) to (16);
	\draw  (13) to (17);
	\draw  (13) to (18);
	
	\draw  (19)[dashed]-- (8);
	
	\draw  (13) to (19);
	
	\draw  (5)--(7);
	\draw  (7) to (8);
	\draw  (7) to (9);
	\draw  (7) to (10);
	\draw  (7) to (11);
	\draw  (7) to (12);

	\draw  (3) to (0);
	\draw  (3) to (1);
	\draw  (3) to (2);
	\draw  (3) to (4);
	\draw  (3) to (5);
	\draw  (3) to (6);
	\end{tikzpicture}
	\normalsize
\end{center}
\caption{Extending weak harmonic labelings through coalescense.}   \label{Figure:Union}
\end{figure}






The other aforementioned construction, which produces exclusively weak harmonic labelings onto $\mathbb{Z}$, is based on the notion of \emph{inner cylinder} of a graph. 

\begin{defi} Given a graph $G$, we define the \emph{inner cylinder of $G$} as the graph $G\breve{\times}\mathbb{Z}$ such that:\begin{itemize}
\item $V_{G\breve{\times}\mathbb{Z}}=\{(v,i)\,:\, v\in V_G, i\in \mathbb{Z}\}$
\item $(v,i)\sim (w,j)$ if and only if ($i=j$ and $v\sim w\in G$) or ($v=w\in V_G\setminus S_G$ and $i=j+1$ or $i=j-1$).
\end{itemize}\end{defi}
%
%
Interestingly, examples of weak harmonic labelings onto $\mathbb{Z}$ can be produced from \emph{any} finite example as the following lemma shows.

\begin{lema}\label{Lemma:ExtendInnerCilinder} A weak harmonic labeling on a finite graph $G$ induces a weak harmonic labeling onto $\mathbb{Z}$ on $G\breve{\times}\mathbb{Z}$.\end{lema}

\begin{proof} Write $|V_G|=n$ and let $\ell:V_G\rightarrow [0,n-1]$ be a weak harmonic labeling. Then, the claimed  labeling $\ell':V_{G\breve{\times}\mathbb{Z}}\rightarrow\mathbb{Z}$ over $G\breve{\times}\mathbb{Z}$ is given by
$$\ell'(v,k)=\ell(v)+kn.$$
\end{proof}

Figure \ref{Figure:InnerCilinder} (Top) shows examples of weak harmonic labelings onto $\mathbb{Z}$ defined using this construction. In some cases we can ``complete" these (weak) infinite examples to harmonic labelings. For instance, the weak harmonic labeling of $K_{1,2}\breve{\times}\mathbb{Z}$ and $K_{1,4}\breve{\times}\mathbb{Z}$ given in Lemma \ref{Lemma:ExtendInnerCilinder} can be extended to an harmonic labeling as it is shown in Figure \ref{Figure:InnerCilinder} (Bottom).

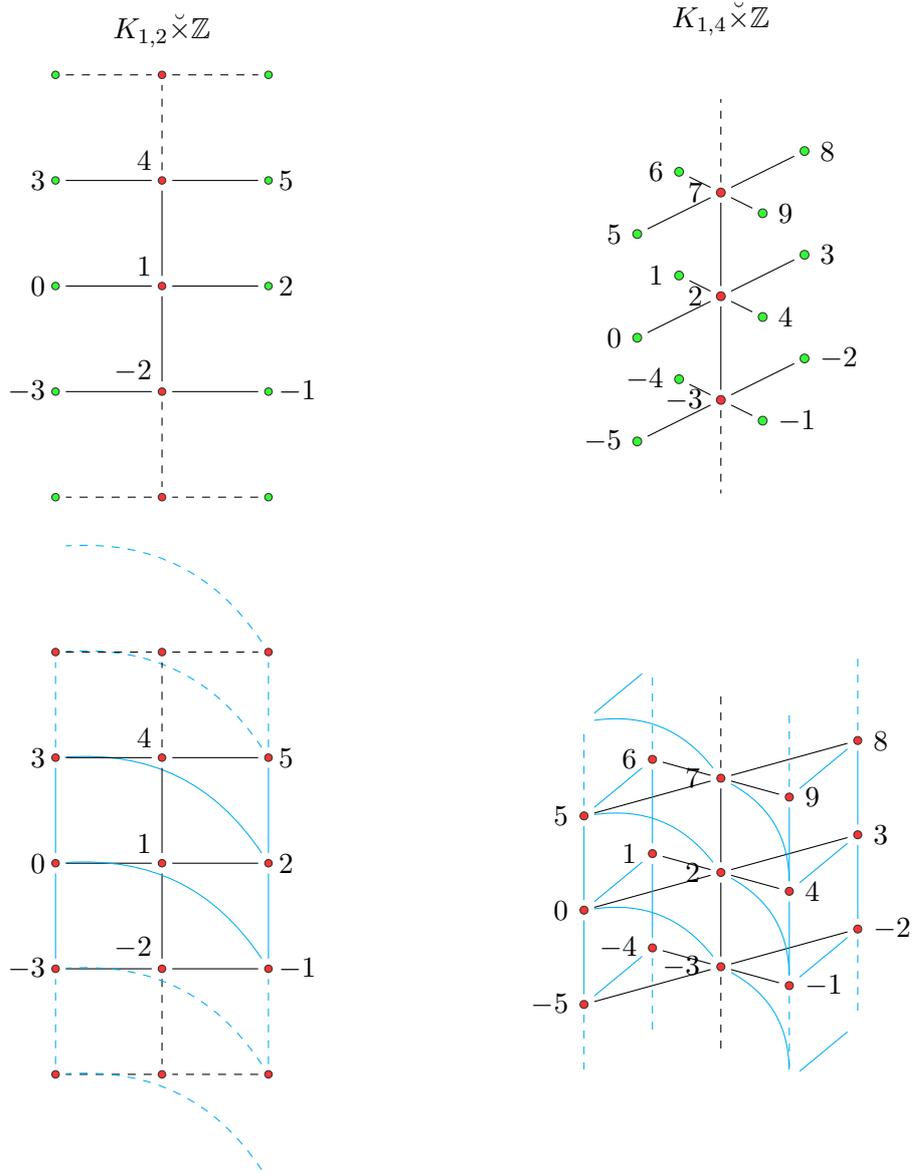
\begin{figure}
\begin{minipage}{0.49\textwidth}
\begin{center}
$K_{1,2}\breve{\times} \mathbb{Z}$\\
\begin{tikzpicture}[scale=0.7]

	\node(0)at(0,0){};
	\filldraw[fill=green!80,draw=black!80] (0) circle (2pt)
	node[left] {$0$}; 
	\node (1)at(2 ,0){};	
	\filldraw[fill=red!80,draw=black!80] (1) circle (2pt)
	node[anchor=south east] {$1$};
	\node (2)at(4,0){};	
	\filldraw[fill=green!80,draw=black!80] (2) circle (2pt)
	node[right] {$2$};
	
	\node (3)at(0,2){};	
	\filldraw[fill=green!80,draw=black!80] (3) circle (2pt)
	node[left] {$3$};
	\node (4)at( 2,2){};	
	\filldraw[fill=red!80,draw=black!80] (4) circle (2pt)
	node[anchor=south east] {$4$};
	\node (5)at(4,2){};	
	\filldraw[fill=green!80,draw=black!80] (5) circle (2pt)
	node[right] {$5$};
	
	\node (6)at(0,4){};	
	\filldraw[fill=green!80,draw=black!80] (6) circle (2pt)
	node[above=2pt] {};	
	\node (7)at(2,4){};	
	\filldraw[fill=red!80,draw=black!80] (7) circle (2pt)
	node[above=2pt] {};
	\node (8)at(4,4){};	
	\filldraw[fill=green!80,draw=black!80] (8) circle (2pt)
	node[above=2pt] {};
	

	\node (-3)at(0,-2){};	
	\filldraw[fill=green!80,draw=black!80] (-3) circle (2pt)
	node[left] {$-3$};
	\node (-2)at(2,-2){};	
	\filldraw[fill=red!80,draw=black!80] (-2) circle (2pt)
	node[anchor=south east] {$-2$};
	\node (-1)at(4,-2){};	
	\filldraw[fill=green!80,draw=black!80] (-1) circle (2pt)
	node[right] {$-1$};
	
	\node (-6)at(0,-4){};	
	\filldraw[fill=green!80,draw=black!80] (-6) circle (2pt)
	node[below=2pt] {};
	\node (-5)at(2,-4){};	
	\filldraw[fill=red!80,draw=black!80] (-5) circle (2pt)
	node[below=2pt] {};
	\node (-4)at(4,-4){};	
	\filldraw[fill=green!80,draw=black!80] (-4) circle (2pt)
	node[below=2pt] {};
	
	
	\draw (0) to (1);
	\draw (1) to (2);
	\draw (3) to (4);
	\draw (4) to (5);
	
	\draw (1) to (4);
	\draw (1) to (-2);
	\draw (-2) [dashed] -- (-5);
	\draw (4) [dashed] -- (7);
	
	\draw (6) [dashed] -- (7);
	\draw (7) [dashed] -- (8);

	\draw (-2) to (-1);
	\draw (-3) to (-2);	
	\draw (-4) [dashed] -- (-5);
	\draw (-5) [dashed] -- (-6);
	
	

	\end{tikzpicture}

%
%
%
%
%
%
	\end{center}
	\end{minipage}\begin{minipage}{0.49\textwidth}
	\begin{center}
		$K_{1,4}\breve{\times} \mathbb{Z}$\\
\begin{tikzpicture}[y=0.5cm, x=0.5cm, scale=0.55]


	\node (10)at(-4,8){};
	\node(11)at(-2,11){};
	\node (12)at(0,10){};
	\node (13)at(4,12){};
	\node (14)at(2,9){};

	\node (5)at(-4,3){};
	\filldraw[fill=green!80,draw=black!80] (5) circle (3pt)    
	node[left=2pt] {$5$};
	\node(6)at(-2,6){};
	\filldraw[fill=green!80,draw=black!80] (6) circle (3pt)   
	 node[left=2pt] {$6$};
	\node (7)at(0,5){};
	\filldraw[fill=red!80,draw=black!80] (7) circle (3pt)   
	 node[left=3pt] {$7$};
	\node (8)at(4,7){};
	\filldraw[fill=green!80,draw=black!80] (8) circle (3pt)   
	 node[right=2pt] {$8$};
	\node (9)at(2,4){};
	\filldraw[fill=green!80,draw=black!80] (9) circle (3pt)   
	 node[right=2pt] {$9$};

	\node (0)at(-4,-2){};
	\filldraw[fill=green!80,draw=black!80] (0) circle (3pt)    
	node[left=2pt] {$0$};
	\node(1)at(-2,1){};
	\filldraw[fill=green!80,draw=black!80] (1) circle (3pt)   
	 node[left=2pt] {$1$};
	\node (2)at(0,0){};
	\filldraw[fill=red!80,draw=black!80] (2) circle (3pt)   
	 node[left=3pt] {$2$};
	\node (3)at(4,2){};
	\filldraw[fill=green!80,draw=black!80] (3) circle (3pt)   
	 node[right=2pt] {$3$};
	\node (4)at(2,-1){};
	\filldraw[fill=green!80,draw=black!80] (4) circle (3pt)   
	 node[right=2pt] {$4$};
	 
\node (-5)at(-4,-7){};
	\filldraw[fill=green!80,draw=black!80] (-5) circle (3pt)    
	node[left=2pt] {$-5$};
	\node(-4)at(-2,-4){};
	\filldraw[fill=green!80,draw=black!80] (-4) circle (3pt)   
	 node[left=2pt] {$-4$};
	\node (-3)at(0,-5){};
	\filldraw[fill=red!80,draw=black!80] (-3) circle (3pt)   
	 node[left=3pt] {$-3$};
	\node (-2)at(4,-3){};
	\filldraw[fill=green!80,draw=black!80] (-2) circle (3pt)   
	 node[right=2pt] {$-2$};
	\node (-1)at(2,-6){};
	\filldraw[fill=green!80,draw=black!80] (-1) circle (3pt)   
	 node[right=2pt] {$-1$};

	\node (-6)at(2,-6){};
	\node (-7)at(4,-8){};
	\node (-8)at(0,-10){};
	\node (-9)at(-2,-9){};
	\node (-10)at(-4,-11){};
	
	\draw (-3)to(-1); 
	\draw (-3)to(-2);
	\draw (-3)to(-4);
	\draw (-3)to(-5);
	
	\draw (2)to(0); 
	\draw (2)to(1);
	\draw (2)to(3);
	\draw (2)to(4);
	
	\draw (7)to(5); 
	\draw (7)to(6);
	\draw (7)to(8);
	\draw (7)to(9);
	
	\draw (2)to(7); 
	\draw (2)to(-3); 
	
	\draw (7)[dashed]--(12); 
	\draw (-3)[dashed]--(-8);

	\end{tikzpicture}
	\end{center}
	\end{minipage}
	
	\begin{minipage}{0.49\textwidth}
\begin{center}
	\begin{tikzpicture}[scale=0.7]

	\node(0)at(0,0){};
	\filldraw[fill=red!80,draw=black!80] (0) circle (2pt)
	node[left] {$0$}; 
	\node (1)at(2 ,0){};	
	\filldraw[fill=red!80,draw=black!80] (1) circle (2pt)
	node[anchor=south east] {$1$};
	\node (2)at(4,0){};	
	\filldraw[fill=red!80,draw=black!80] (2) circle (2pt)
	node[right] {$2$};
	
	\node (3)at(0,2){};	
	\filldraw[fill=red!80,draw=black!80] (3) circle (2pt)
	node[left] {$3$};
	\node (4)at( 2,2){};	
	\filldraw[fill=red!80,draw=black!80] (4) circle (2pt)
	node[anchor=south east] {$4$};
	\node (5)at(4,2){};	
	\filldraw[fill=red!80,draw=black!80] (5) circle (2pt)
	node[right] {$5$};
	
	\node (6)at(0,4){};	
	\filldraw[fill=red!80,draw=black!80] (6) circle (2pt)
	node[above=2pt] {};	
	\node (7)at(2,4){};	
	\filldraw[fill=red!80,draw=black!80] (7) circle (2pt)
	node[above=2pt] {};
	\node (8)at(4,4){};	
	\filldraw[fill=red!80,draw=black!80] (8) circle (2pt)
	node[above=2pt] {};
	
	\node (9)at(0,6){};	

	\node (-3)at(0,-2){};	
	\filldraw[fill=red!80,draw=black!80] (-3) circle (2pt)
	node[left] {$-3$};
	\node (-2)at(2,-2){};	
	\filldraw[fill=red!80,draw=black!80] (-2) circle (2pt)
	node[anchor=south east] {$-2$};
	\node (-1)at(4,-2){};	
	\filldraw[fill=red!80,draw=black!80] (-1) circle (2pt)
	node[right] {$-1$};
	
	\node (-6)at(0,-4){};	
	\filldraw[fill=red!80,draw=black!80] (-6) circle (2pt)
	node[below=2pt] {};
	\node (-5)at(2,-4){};	
	\filldraw[fill=red!80,draw=black!80] (-5) circle (2pt)
	node[below=2pt] {};
	\node (-4)at(4,-4){};	
	\filldraw[fill=red!80,draw=black!80] (-4) circle (2pt)
	node[below=2pt] {};
	
	\node (-7)at(4,-6){};	
	
	\draw (0) to (1);
	\draw (1) to (2);
	\draw (2) [bend right, cyan]to (3);
	\draw (3) to (4);
	\draw (4) to (5);
	\draw (5)[bend right, cyan, dashed]to  (6);
	
	\draw (1) to (4);
	\draw (1) to (-2);
	\draw (-2) [dashed] -- (-5);
	\draw (4) [dashed] -- (7);
	
	\draw (6) [dashed] -- (7);
	\draw (7) [dashed] -- (8);
	\draw (8) [bend right, cyan, dashed]to (9);

	\draw (-1) [bend right, cyan]to (0);
	\draw (-2) to (-1);
	\draw (-3) to (-2);	
	\draw (-4) [bend right, cyan, dashed]to (-3);
	\draw (-4) [dashed] -- (-5);
	\draw (-5) [dashed] -- (-6);
	\draw (-7)  [bend right, cyan, dashed]to (-6);
	
	\draw [color=cyan](0) to (3);	
	\draw [color=cyan](-3) to (0);
	\draw [color=cyan](-6) [dashed]-- (-3);	
	\draw [color=cyan](3) [dashed]-- (6);
	
	\draw [color=cyan](-4)[dashed]--(-1);	
	\draw [color=cyan](-1) to (2);
	\draw [color=cyan](2) to (5);	
	\draw [color=cyan](5) [dashed]-- (8);

	\end{tikzpicture}
	\end{center}
\end{minipage}\begin{minipage}{0.49\textwidth}
\begin{center}
	\begin{tikzpicture}[y=0.5cm, x=0.9cm, scale=0.5]


	\node (10)at(-4,8){};
	\node(11)at(-2,11){};
	\node (12)at(0,10){};
	\node (13)at(4,12){};
	\node (14)at(2,9){};

	\node (5)at(-4,3){};
	\filldraw[fill=red!80,draw=black!80] (5) circle (3pt)    
	node[left=2pt] {$5$};
	\node(6)at(-2,6){};
	\filldraw[fill=red!80,draw=black!80] (6) circle (3pt)   
	 node[left=2pt] {$6$};
	\node (7)at(0,5){};
	\filldraw[fill=red!80,draw=black!80] (7) circle (3pt)   
	 node[left=4pt] {$7$};
	\node (8)at(4,7){};
	\filldraw[fill=red!80,draw=black!80] (8) circle (3pt)   
	 node[right=2pt] {$8$};
	\node (9)at(2,4){};
	\filldraw[fill=red!80,draw=black!80] (9) circle (3pt)   
	 node[right=2pt] {$9$};

	\node (0)at(-4,-2){};
	\filldraw[fill=red!80,draw=black!80] (0) circle (3pt)    
	node[left=2pt] {$0$};
	\node(1)at(-2,1){};
	\filldraw[fill=red!80,draw=black!80] (1) circle (3pt)   
	 node[left=2pt] {$1$};
	\node (2)at(0,0){};
	\filldraw[fill=red!80,draw=black!80] (2) circle (3pt)   
	 node[left=4pt] {$2$};
	\node (3)at(4,2){};
	\filldraw[fill=red!80,draw=black!80] (3) circle (3pt)   
	 node[right=2pt] {$3$};
	\node (4)at(2,-1){};
	\filldraw[fill=red!80,draw=black!80] (4) circle (3pt)   
	 node[right=2pt] {$4$};
	 
\node (-5)at(-4,-7){};
	\filldraw[fill=red!80,draw=black!80] (-5) circle (3pt)    
	node[left=2pt] {$-5$};
	\node(-4)at(-2,-4){};
	\filldraw[fill=red!80,draw=black!80] (-4) circle (3pt)   
	 node[left=2pt] {$-4$};
	\node (-3)at(0,-5){};
	\filldraw[fill=red!80,draw=black!80] (-3) circle (3pt)   
	 node[left=4pt] {$-3$};
	\node (-2)at(4,-3){};
	\filldraw[fill=red!80,draw=black!80] (-2) circle (3pt)   
	 node[right=2pt] {$-2$};
	\node (-1)at(2,-6){};
	\filldraw[fill=red!80,draw=black!80] (-1) circle (3pt)   
	 node[right=2pt] {$-1$};

	\node (-6)at(2,-11){};
	\node (-7)at(4,-8){};
	\node (-8)at(0,-10){};
	\node (-9)at(-2,-9){};
	\node (-10)at(-4,-11){};

	\draw (-3)to(-1); 
	\draw (-3)to(-2);
	\draw (-3)to(-4);
	\draw (-3)to(-5);
	\draw [color=cyan](-2)to(-1); 
	\draw [color=cyan](-4)to(-5);
	
	\draw (2)to(0); 
	\draw (2)to(1);
	\draw (2)to(3);
	\draw (2)to(4);
	\draw [color=cyan](0)to(1); 
	\draw [color=cyan](3)to(4);
	
	\draw (7)to(5); 
	\draw (7)to(6);
	\draw (7)to(8);
	\draw (7)to(9);
	\draw [color=cyan](5)to(6); 
	\draw [color=cyan](8)to(9);
	
	\draw (2)to(7); 
	\draw (2)to(-3); 
	\draw [color=cyan](0)to(5); 
	\draw [color=cyan](0)to(-5);
	\draw [color=cyan](1)to(6); 
	\draw [color=cyan](1)to(-4);
	\draw [color=cyan](3)to(8); 
	\draw [color=cyan](3)to(-2);
	\draw [color=cyan](4)to(9); 
	\draw [color=cyan](4)to(-1);

	\draw (7)[dashed]--(12); 
	\draw (-3)[dashed]--(-8); 
	\draw [color=cyan](5)[dashed]--(10); 
	\draw [color=cyan](6)[dashed]--(11); 
	\draw [color=cyan](8)[dashed]--(13); 
	\draw [color=cyan](9)[dashed]--(14); 
	\draw [color=cyan](-2)[dashed]--(-7); 
	\draw [color=cyan](-1)[dashed]--(-6); 
	\draw [color=cyan](-4)[dashed]--(-9); 
	\draw [color=cyan](-5)[dashed]--(-10);

	\draw[color=cyan] (10) to[bend left] (7);
	\draw[color=cyan] (7) to[bend left] (4);
	\draw[color=cyan] (5) to[bend left] (2);
	\draw[color=cyan] (2) to[bend left] (-1);
	\draw[color=cyan] (0) to[bend left] (-3);
	\draw[color=cyan] (-3) to[bend left] (-6);
	\draw[color=cyan] (10) to (11);
	\draw[color=cyan] (-6) to (-7);
	\end{tikzpicture}
		\end{center}
\end{minipage}
	\caption{\textbf{Top.} The weak harmonic labeling induced in the inner cylinder of $K_{1,2}$ (left) and $K_{1,4}$ (right). \textbf{Bottom.} Harmonic labeling from the weak labeling of $K_{1,2}\breve{\times}\mathbb{Z}$ (left) and $K_{1,4}\breve{\times}\mathbb{Z}$ (right). The cyan colored edges represent added edges to the original weak labelings.}   \label{Figure:InnerCilinder}	
	\end{figure}

\subsection*{Labelings inferred from finite models} The weakly labeled graph in Figure \ref{Figure:CkGeneralized} (bottom) is a particular case of the family portrayed in Figure \ref{Figure:cuadriculado}, which we call $C^{k,h}$. We note that this collection can too be extended to $[0,\infty]$ and $\mathbb{Z}$, and that this last extension produces an harmonically labeled graph, $C^{k,\infty}$. Formally, $V_{C^{k,\infty}}=\mathbb{Z}$ and $E_{C^{k,\infty}}=\{\{a,b\}\,:\,b=a-1,a+1,a+k,a-k\}$. This new example of harmonic labeling is indeed part of a far more general family. Note that for $b\nsim a$ we can add the edges $(s+1)(b-a)+a\sim s(b-a)+a$ for each $s\in\mathbb{Z}$ and obtain a new harmonically labeled graph (see Figure \ref{Figure:ChkstrongCyan}). We can repeat this process to the newly generated example to obtain infinitely many new ones (a different for each edges selected for addition and each $k$). We make this construction precise next.

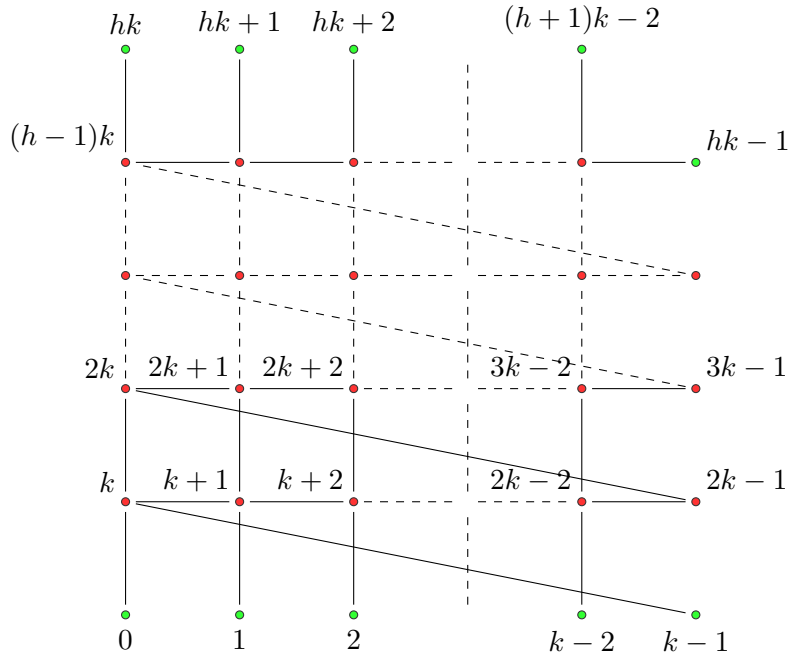
\begin{figure}
\begin{tikzpicture}[scale=0.75]

	\node(0)at(0,0){};
	\filldraw[fill=green!80,draw=black!80] (0) circle (2pt)
	node[below=2pt] {$0$}; 
	\node (1)at(2,0){};	
	\filldraw[fill=green!80,draw=black!80] (1) circle (2pt)
	node[below=2pt] {$1$}; 
	\node (2)at(4,0){};
	\filldraw[fill=green!80,draw=black!80] (2) circle (2pt)
	node[below=2pt] {$2$}; 
	\node (3)at(6,0){};
	\node(4)at(8,0){};
	\filldraw[fill=green!80,draw=black!80] (4) circle (2pt)
	node[below=2pt] {$k-2$}; 
	\node (5)at(10,0){};
	\filldraw[fill=green!80,draw=black!80] (5) circle (2pt)
	node[below=2pt] {$k-1$};

	\node (6)at(0,2){};
	\filldraw[fill=red!80,draw=black!80] (6) circle (2pt)
	node[anchor=south east] {$k$}; 
	\node(7)at(2,2){};
	\filldraw[fill=red!80,draw=black!80] (7) circle (2pt)
	node[anchor=south east] {$k+1$}; 	
	\node(8)at(4,2){};
	\filldraw[fill=red!80,draw=black!80] (8) circle (2pt)
	node[anchor=south east] {$k+2$}; 	
	\node (9)at(6,2){};
	\node (10)at(8,2){};
	\filldraw[fill=red!80,draw=black!80] (10) circle (2pt)
	node[anchor=south east] {$2k-2$};
	\node (11)at(10,2){};
	\filldraw[fill=red!80,draw=black!80] (11) circle (2pt)
	node[anchor=south west] {$2k-1$};

	\node (12)at(0,4){};
	\filldraw[fill=red!80,draw=black!80] (12) circle (2pt)
	node[anchor=south east] {$2k$};
	\node (13)at(2,4){};
	\filldraw[fill=red!80,draw=black!80] (13) circle (2pt)
	node[anchor=south east] {$2k+1$};	
	\node(14)at(4,4){};
	\filldraw[fill=red!80,draw=black!80] (14) circle (2pt)
	node[anchor=south east] {$2k+2$};
	\node (15)at(6,4){};
	\node (16)at(8,4){};
	\filldraw[fill=red!80,draw=black!80] (16) circle (2pt)
	node[anchor=south east] {$3k-2$};
	\node (17)at(10,4){};
	\filldraw[fill=red!80,draw=black!80] (17) circle (2pt)
	node[anchor=south west] {$3k-1$};
	
	\node (18)at(0,6){};
	\filldraw[fill=red!80,draw=black!80] (18) circle (2pt)
	node[anchor=south east] {};
	\node(19)at(2,6){};
	\filldraw[fill=red!80,draw=black!80] (19) circle (2pt)
	node[anchor=south east] {};	
	\node(20)at(4,6){};
	\filldraw[fill=red!80,draw=black!80] (20) circle (2pt)
	node[anchor=south east] {};
	\node (21)at(6,6){};
	\node (22)at(8,6){};
	\filldraw[fill=red!80,draw=black!80] (22) circle (2pt)
	node[anchor=south east] {};
	\node(23)at(10,6){};
	\filldraw[fill=red!80,draw=black!80] (23) circle (2pt)
	node[anchor=south west] {};

	\node (24)at(0,8){};
	\filldraw[fill=red!80,draw=black!80] (24) circle (2pt)
	node[anchor=south east] {$(h-1)k$};
	\node(25)at(2,8){};
	\filldraw[fill=red!80,draw=black!80] (25) circle (2pt)
	node[above=2pt] {};	
	\node(26)at(4,8){};
	\filldraw[fill=red!80,draw=black!80] (26) circle (2pt)
	node[above=2pt] {};
	\node (27)at(6,8){};
	\node (28)at(8,8){};
	\filldraw[fill=red!80,draw=black!80] (28) circle (2pt)
	node[above=2pt] {};
	\node (29)at(10,8){};
	\filldraw[fill=green!80,draw=black!80] (29) circle (2pt)
	node[anchor=south west] {$hk-1$};	
	
	\node (30)at(0,10){};
	\filldraw[fill=green!80,draw=black!80] (30) circle (2pt)
	node[above=2pt] {$hk$};
	\node(31)at(2,10){};
	\filldraw[fill=green!80,draw=black!80] (31) circle (2pt)
	node[above=2pt] {$hk+1$};	
	\node(32)at(4,10){};
	\filldraw[fill=green!80,draw=black!80] (32) circle (2pt)
	node[above=2pt] {$hk+2$};
	\node (33)at(6,10){};
	\node (34)at(8,10){};
	\filldraw[fill=green!80,draw=black!80] (34) circle (2pt)
	node[above=2pt] {$(h+1)k-2$};
	
	\draw (0) to (6);
	\draw (6) to (12);
	\draw (12) [dashed] -- (18);

	\draw (1) to (7);
	\draw (7) to (13);
    \draw (13) [dashed] -- (19); 
      	
	\draw (2) to (8);
	\draw (8) to (14);
	\draw (14) [dashed] -- (20);
	
	\draw (3) [dashed] -- (9);
	\draw (9) [dashed] -- (15);
	\draw (15)[dashed] -- (21);
	
	\draw (4) to (10);
	\draw (10) to (16);
	\draw (16) [dashed] -- (22);	
	
	\draw (5) to (6);
	\draw (11) to (12);
	
	\draw (6) to (7);
	\draw (7) to (8);
	\draw (8) [dashed] -- (9);
	\draw (9) [dashed] -- (10);
	\draw (10) to (11);
	
	\draw (12) to (13);
	\draw (13) to (14);
	\draw (14) [dashed]--(15);
	\draw (15) [dashed] --(16);
	\draw (16) to (17);
	
	\draw (24) -- (25);
	\draw (25) -- (26);
	\draw (26) [dashed]-- (27);
	\draw (27) [dashed]-- (28);
	\draw (28) -- (29);	
	
	\draw (18) [dashed] -- (19);
	\draw (19) [dashed] -- (20);
	\draw (20) [dashed]-- (21);
	\draw (21) [dashed]-- (22);	
	\draw (22) [dashed] -- (23);	
	
	\draw (18) [dashed] -- (24);
	\draw (19) [dashed] -- (25);
	\draw (20) [dashed] -- (26);
	\draw (21) [dashed]-- (27);
	\draw (22) [dashed] -- (28);

	\draw (24) -- (30);
	\draw (25) -- (31);
	\draw (26) -- (32);
	\draw (27) [dashed]-- (33);
	\draw (28) -- (34);	
		
	\draw (17) [dashed]-- (18);	
	\draw (23) [dashed] -- (24);

	\end{tikzpicture}
\caption{The family $C^{k,h}$ of non-acyclic weakly labeled graphs which generalizes the family of Figure \ref{Figure:CkGeneralized} (bottom).}
\label{Figure:cuadriculado}
\end{figure}

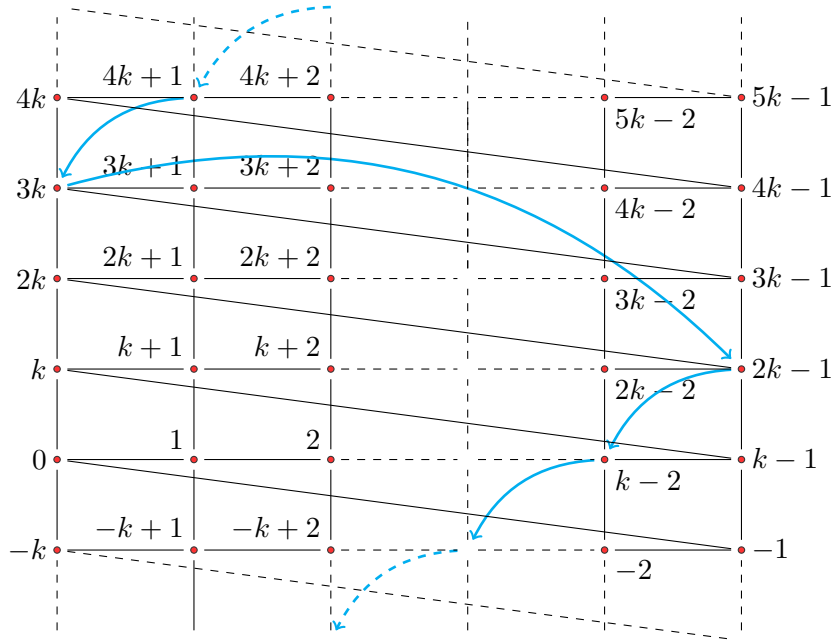
\begin{figure}

\begin{tikzpicture}[scale=0.6,rotate=0]

	
	\node (-7)at(0,-4){};
	\node (-8)at(3,-4){};	
	\node (-9)at(6,-4){};
	\node (-10)at(9,-4){};
	\node (-11)at(12,-4){};
	\node (-12)at(15,-4){};

	\node (-1)at(0,-2){};
	\filldraw[fill=red!80,draw=black!80] (-1) circle (2pt)
	node[left] {$-k$}; 
	\node (-2)at(3,-2){};	
	\filldraw[fill=red!80,draw=black!80] (-2) circle (2pt)
	node[anchor=south east] {$-k+1$};
	\node(-3)at(6,-2){};
	\filldraw[fill=red!80,draw=black!80] (-3) circle (2pt)
	node[anchor=south east] {$-k+2$};
	\node (-4)at(9,-2){};
	\node(-5)at(12,-2){};
	\filldraw[fill=red!80,draw=black!80] (-5) circle (2pt)
	node[anchor=north west] {$-2$};
	\node(-6)at(15,-2){};	
	\filldraw[fill=red!80,draw=black!80] (-6) circle (2pt)
	node[right] {$-1$};

	\node (0)at(0,0){};
	\filldraw[fill=red!80,draw=black!80] (0) circle (2pt)
	node[left] {$0$};
	\node (1)at(3,0){};	
	\filldraw[fill=red!80,draw=black!80] (1) circle (2pt)
	node[anchor=south east] {$1$};
	\node(2)at(6,0){};
	\filldraw[fill=red!80,draw=black!80] (2) circle (2pt)
	node[anchor=south east] {$2$};
	\node (3)at(9,0){};
	\node(4)at(12,0){};
	\filldraw[fill=red!80,draw=black!80] (4) circle (2pt)
	node[anchor=north west] {$k-2$};
	\node(5)at(15,0){};
	\filldraw[fill=red!80,draw=black!80] (5) circle (2pt)
	node[right] {$k-1$};

	\node(6)at(0,2){};
	\filldraw[fill=red!80,draw=black!80] (6) circle (2pt)
	node[left] {$k$};
	\node (7)at(3,2){};	
	\filldraw[fill=red!80,draw=black!80] (7) circle (2pt)
	node[anchor=south east] {$k+1$};
	\node(8)at(6,2){};
	\filldraw[fill=red!80,draw=black!80] (8) circle (2pt)
	node[anchor=south east] {$k+2$};
	\node (9)at(9,2){};
	\node(10)at(12,2){};
	\filldraw[fill=red!80,draw=black!80] (10) circle (2pt)
	node[anchor=north west] {$2k-2$};
	\node(11)at(15,2){};
	\filldraw[fill=red!80,draw=black!80] (11) circle (2pt)
	node[right] {$2k-1$};

	\node (12)at(0,4){};
	\filldraw[fill=red!80,draw=black!80] (12) circle (2pt)
	node[left] {$2k$};
	\node (13)at(3,4){};
	\filldraw[fill=red!80,draw=black!80] (13) circle (2pt)
	node[anchor=south east] {$2k+1$};	
	\node(14)at(6,4){};
	\filldraw[fill=red!80,draw=black!80] (14) circle (2pt)
	node[anchor=south east] {$2k+2$};
	\node (15)at(9,4){};
	\node (16)at(12,4){};
	\filldraw[fill=red!80,draw=black!80] (16) circle (2pt)
	node[anchor=north west] {$3k-2$};
	\node (17)at(15,4){};
	\filldraw[fill=red!80,draw=black!80] (17) circle (2pt)
	node[right] {$3k-1$};
	
	\node (18)at(0,6){};
	\filldraw[fill=red!80,draw=black!80] (18) circle (2pt)
	node[left] {$3k$};
	\node (19)at(3,6){};
	\filldraw[fill=red!80,draw=black!80] (19) circle (2pt)
	node[anchor=south east] {$3k+1$};	
	\node (20)at(6,6){};
	\filldraw[fill=red!80,draw=black!80] (20) circle (2pt)
	node[anchor=south east] {$3k+2$};
	\node (21)at(9,6){};
	\node (22)at(12,6){};
	\filldraw[fill=red!80,draw=black!80] (22) circle (2pt)
	node[anchor=north west] {$4k-2$};
	\node(23)at(15,6){};
	\filldraw[fill=red!80,draw=black!80] (23) circle (2pt)
	node[right] {$4k-1$};
	
	\node (24)at(0,8){};
	\filldraw[fill=red!80,draw=black!80] (24) circle (2pt)
	node[left] {$4k$};
	\node(25)at(3,8){};	
	\filldraw[fill=red!80,draw=black!80] (25) circle (2pt)
	node[anchor=south east] {$4k+1$};
	\node (26)at(6,8){};
	\filldraw[fill=red!80,draw=black!80] (26) circle (2pt)
	node[anchor=south east] {$4k+2$};
	\node (27)at(9,8){};
	\node(28)at(12,8){};
	\filldraw[fill=red!80,draw=black!80] (28) circle (2pt)
	node[anchor=north west] {$5k-2$};
	\node (29)at(15,8){};
	\filldraw[fill=red!80,draw=black!80] (29) circle (2pt)
	node[right] {$5k-1$};
	
	\draw[->,dashed,color=cyan,line width = 1pt] (6,10) to[bend right] (25);
	\draw[->,color=cyan,line width = 1pt] (25) to[bend right] (18);
	\draw[->,color=cyan,line width = 1pt] (18) to[bend left] (11);
	\draw[->,color=cyan,line width = 1pt] (11) to[bend right] (4);
	\draw[->,color=cyan,line width = 1pt] (4) to[bend right] (-4);
	\draw[->,dashed,color=cyan,line width = 1pt] (-4) to[bend right] (-9);
	
	\node (30)at(0,10){};
	\node (31)at(3,10){};	
	\node (32)at(6,10){};
	\node (33)at(9,10){};
	\node (34)at(12,10){};
	\node (35)at(15,10){};

	\draw (-7) [dashed]to (-1);
	\draw (-1)to(0);
	\draw (0) to (6);
	\draw (6) to (12);
	\draw (12) to (18);
	\draw (18) to (24);
	\draw (24)[dashed] to (30);

	\draw (-8)to(-2);
	\draw (-2) to (1);
	\draw (1) to (7);
	\draw (7) to (13);
    \draw (13) to (19); 
    \draw (19) to (25);
	\draw (25) [dashed]to (31);
	
	\draw (-9)[dashed]to(-3);
	\draw (-3) to (2);      	
	\draw (2) to (8);
	\draw (8) to (14);
	\draw (14) to (20);
	\draw (20) to (26);
	\draw (26)[dashed] to (32);

	\draw (-10)[dashed]--(-4);
	\draw (-4)[dashed] -- (3); 
	\draw (3) [dashed] -- (9);
	\draw (9) [dashed] -- (15);
	\draw (15)[dashed] -- (21);
	\draw (21)[dashed] -- (27);
	\draw (15)[dashed] -- (33);
	
	\draw (-11)[dashed]to(-5);
	\draw (-5) to (4); 
	\draw (4) to (10);
	\draw (10) to (16);
	\draw (16) to (22);	
	\draw (22) to (28);
	\draw (28)[dashed] to (34);	
	
	\draw (29)[dashed]to(35);
	\draw (23)to(29);
	\draw (17) to (23); 
	\draw (11) to (17);
	\draw (5) to (11);
	\draw (-6) to (5);	
	\draw (-12)[dashed] to (-6);

	\draw (5) to (6); 
	\draw (11) to (12);
	\draw (17) to (18);	
	\draw (23) to (24);
	\draw (29) [dashed]-- (30);
	\draw (-6) to (0);
	\draw (-1) [dashed]-- (-12);

	\draw (-1) to (-2); 
	\draw (-2) to (-3);
	\draw (-3) [dashed]-- (-4);
	\draw (-4) [dashed]-- (-5);
	\draw (-5) to (-6);	
	
	\draw (0) to (1); 
	\draw (1) to (2);
	\draw (2) [dashed]-- (3);
	\draw (3) [dashed]-- (4);
	\draw (4) to (5);

	\draw (6) to (7); 
	\draw (7) to (8);
	\draw (8) [dashed] -- (9);
	\draw (9) [dashed] -- (10);
	\draw (10) to (11);
	
	\draw (12) to (13);
	\draw (13) to (14);
	\draw (14) [dashed]--(15);
	\draw (15) [dashed] --(16);
	\draw (16) to (17);
	
	\draw (18) to (19);
	\draw (19) to (20);
	\draw (20) [dashed]--(21);
	\draw (21) [dashed] --(22);
	\draw (22) to (23);
	
	\draw (24) to (25);
	\draw (25) to (26);
	\draw (26) [dashed]--(27);
	\draw (27) [dashed] --(28);
	\draw (28) to (29);
	
	\end{tikzpicture}

	\caption{New harmonic labeling from $C^{k,\infty}$.} \label{Figure:ChkstrongCyan}
\end{figure}

Let $\mathcal{B}=\{(i,k)\, : \, k>1 \text{ and } 0\leq i\leq k-1\}$. For any (finite or infinite) subset $B$ of $\mathcal{B}$ we form the graph $P_B$ obtained from (the harmonically labeled graph) $\mathbb{Z}$ by adding the edges $\{(s+1)k+i,sk+i\}$ for every $s\in\mathbb{Z}$ for each $(i,k)\in B$. We call $B$ a \emph{base} for $P_B$ and we write $P_B=\langle x\, : \, x\in B\rangle$ (the elements of $B$ are the \emph{spanning edges} of $P_B$). We picture a concrete example in Figure \ref{Figure:P_B}.

\begin{prop} For any $B\subset \mathcal{B}$, $P_B$ is an  harmonically labeled graph. Furthermore, $P_B=P_{B'}$ if and only if $B=B'$.\end{prop}

\begin{proof} First of all, we note that the set of edges added by different pairs $(i,k)$ and $(i',k')$ are disjoint. Indeed, the system $$\begin{cases}sk+i=s'k'+i'\\(s+1)k+i=(s'+1)k'+i'\end{cases}$$ has unique solution $s=s'$, $k=k'$ and $i=i'$ for $0\leq i,i'\leq k-1$. So it suffices to show that if a vertex $v$ is harmonically labeled then adding the edges $\{(s+1)k+i,sk+i\}$ to a $P_{B'}$ corresponding to a single member $(i,k)\in B\setminus B'$ keeps $v$ harmonic. This is clear if the vertex $v$ is not incident to any of the added edges. Otherwise, $v$ has new adjacent vertices labeled $\ell(v)-k$ and $\ell(v)+k$. Therefore $$\sum_{w\sim v\in P_{B'}}\ell(w)+(\ell(v)-k)+(\ell(v)+k)=(\deg(v)+2)\ell(v),$$ which proves the claim. Finally, by the previous remarks, every edge is exclusive of a given $(i,k)$ with $k\geq 2$ and $0\leq i\leq k-1$. Therefore, $P_B=P_{B'}$ if and only if $B=B'$.\end{proof}

\begin{coro} The collection $\mathcal{P}=\{P_B\, : \, B\subset\mathcal{B}\}$ is a non-numerable family of harmonically labeled graphs.\end{coro}

\begin{figure}
\begin{tikzpicture}[scale=0.9]

	\node (1)at(0.5 ,0){};	
	\node (2)at(2 ,0){};
	\filldraw[fill=red!80,draw=black!80] (2) circle (2pt)
	node[below=2pt] {$-2$};
	\node (3)at(3 ,0){};	
	\filldraw[fill=red!80,draw=black!80] (3) circle (2pt)
	node[below=2pt] {$-1$};
	\node (4)at(4,0){};	
	\filldraw[fill=red!80,draw=black!80] (4) circle (2pt)
	node[below=2pt] {$0$};
	\node (5)at(5 ,0){};	
	\filldraw[fill=red!80,draw=black!80] (5) circle (2pt)
	node[below=2pt] {$1$};
	\node (6)at(6,0){};	
	\filldraw[fill=red!80,draw=black!80] (6) circle (2pt)
	node[below=2pt] {$2$};
	\node (7)at(7 ,0){};	
	\filldraw[fill=red!80,draw=black!80] (7) circle (2pt)
	node[below=2pt] {$3$};
	\node (8)at( 8,0){};	
	\filldraw[fill=red!80,draw=black!80] (8) circle (2pt)
	node[below=2pt] {$4$};
	\node (9)at(9 ,0){};	
	\filldraw[fill=red!80,draw=black!80] (9) circle (2pt)
	node[below=2pt] {$5$};
	\node (10)at(10,0){};	
	\filldraw[fill=red!80,draw=black!80] (10) circle (2pt)
	node[below=2pt] {$6$};
	\node (11)at(11 ,0){};	
	\filldraw[fill=red!80,draw=black!80] (11) circle (2pt)
	node[below=2pt] {$7$};
	\node (12)at(12,0){};	
	\filldraw[fill=red!80,draw=black!80] (12) circle (2pt)
	node[below=2pt] {$8$};	
	\node (13)at(13 ,0){};	
	\filldraw[fill=red!80,draw=black!80] (13) circle (2pt)
	node[below=2pt] {$9$};
	\node (14)at(14,0){};	
	\filldraw[fill=red!80,draw=black!80] (14) circle (2pt)
	node[below=2pt] {$10$};
	\node (15)at(15.5 ,0){};	

	\draw  (1) [dashed] -- (2);
	\draw (2) to (3);
	\draw (3) to (4);
	\draw (4) to (5);
	\draw (5) to (6);
	\draw (6) to (7);
	\draw (7) to (8);
	\draw (8) to (9);
	\draw (9) to (10);
	\draw (10) to (11);
	\draw (11) to (12);
	\draw (12) to (13);
	\draw (13) to (14);
	\draw (14) [dashed] -- (15);

	\draw[dashed] (0,0) to[bend left] (2);
	\draw (2) to[bend left] (4);
	\draw[color=cyan,line width = 1pt] (4) to[bend left] (6);
	\draw (6) to[bend left] (8);
	\draw (8) to[bend left] (10);
	\draw (10) to[bend left] (12);
	\draw (12) to[bend left] (14);
	\draw[dashed] (14) to[bend left] (16,0);
	
	\draw (2) to[bend left] (5);	
	\draw[color=cyan,line width = 1pt] (5) to[bend left] (8);
	\draw (8) to[bend left] (11);
 	\draw (11) to[bend left] (14);
	
	\draw (7)[color=cyan,line width = 1pt] to[bend left] (12);	
    \draw (2) to[bend left] (7);
    \draw[dashed] (12) to[bend left] (16,0.3);

	\draw[dashed] (14)  to[bend left] (16,0.3);
 \draw[dashed] (2)  to[bend right] (0,0.3);
 \draw[dashed] (2)  to[bend right] (0,0.6);




	\end{tikzpicture}
	\caption{The harmonically labeled graph $\langle (0,2), (1,3), (3,5)\rangle$ (in cyan, the spanning edges).}
	\label{Figure:P_B}
\end{figure}
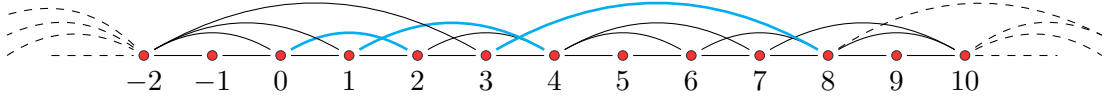

Some of the previously presented examples actually belong to the collection $\mathcal{P}$. For example, $C^{k,\infty}=\langle (0,k) \rangle$ and $K_{1,2}\breve{\times}\mathbb{Z}=\langle (1,3) \rangle$. However, $K_{1,4}\breve{\times}\mathbb{Z}$ is not one of these graphs.

\begin{obs}\label{Remark:OpenProblems} A set $V'\subset V_G$ is said to be a \textit{labeling spanning set} if the values of a labeling $\ell$ on the vertices of $V'$ completely determines the labeling of $G$ (by the harmonic property). In \cite[\S 6]{BCPT} the authors ask which connected graphs other than $\mathbb{Z}$ admit an harmonic labeling spanned by a finite set. We claim that the members $\langle (0,k) \rangle$ of $P_{\mathcal{B}}$ for any $k\in\mathbb{Z}$ are  finitely spanned by vertices labeled $0$ and $1$. Indeed, these two labels trivially determine all labels from $0$ to $k$. The labels $x_{k+1},x_{k+2},\dots,x_{2k}$ pictured in Figure \ref{Figure:P_B_Remark} are solutions of the system
$$\left(\begin{matrix}
1&0&0&\ldots&0&0&1\\
2&-1&0&\dots&0&0&0\\
-1&2&-1&\ldots&0&0&0\\
\vdots&\vdots&\vdots&\ldots&\vdots&\vdots&\vdots\\
0&0&0&\ldots&-1&2&-1\\
\end{matrix}\right)\cdot \left(\begin{matrix}
x_{k+1}\\
x_{k+2}\\
x_{k+3}\\
\vdots\\
x_{2k}\\
\end{matrix}\right)=\left(\begin{matrix}
3k+1\\
k\\
0\\
\vdots\\
0\\
\end{matrix}\right)$$
whose matrix is non-singular for every $k\in\mathbb{Z}$. The claim is then settled by an inductive argument. 
\begin{figure}
\begin{tikzpicture}[scale=0.9]

	\node (1)at(0.5 ,0){};	
	\node (2)at(2 ,0){};
	\filldraw[fill=red!80,draw=black!80] (2) circle (2pt)
	node[below=2pt] {};
	\node (3)at(3 ,0){};	
	\filldraw[fill=red!80,draw=black!80] (3) circle (2pt)
	node[below=2pt] {};
	\node (4)at(4,0){};	
	\filldraw[fill=red!80,draw=black!80] (4) circle (2pt)
	node[below=2pt] {\small $0$};
	\node (5)at(5 ,0){};	
	\filldraw[fill=red!80,draw=black!80] (5) circle (2pt)
	node[below=2pt] {\small $1$};
	\node (6)at(6,0){};	
	\filldraw[fill=red!80,draw=black!80] (6) circle (2pt)
	node[below=2pt] {\small $2$};
	\node (8)at( 8,0){};	
	\filldraw[fill=red!80,draw=black!80] (8) circle (2pt)
	node[below=2pt] {\small $k-1$};
	\node (9)at(9 ,0){};	
	\filldraw[fill=red!80,draw=black!80] (9) circle (2pt)
	node[below=2pt] {\small $k$};
	\node (10)at(10,0){};	
	\filldraw[fill=red!80,draw=black!80] (10) circle (2pt)
	node[below=2pt] {\small $x_{k+1}$};
	\node (11)at(11 ,0){};	
	\filldraw[fill=red!80,draw=black!80] (11) circle (2pt)
	node[below=2pt] {\small $x_{k+2}$};
	\node (13)at(13 ,0){};	
	\filldraw[fill=red!80,draw=black!80] (13) circle (2pt)
	node[below=2pt] {\small $x_{2k-1}$};
	\node (14)at(14,0){};	
	\filldraw[fill=red!80,draw=black!80] (14) circle (2pt)
	node[below=2pt] {\small $x_{2k}$};
	\node (15)at(15.5 ,0){};	

	\draw  (1) [dashed] -- (2);
	\draw (2) to (3);
	\draw (3) to (4);
	\draw (4) to (5);
	\draw (5) to (6);
	\draw[dashed] (6) to (8);
	\draw (8) to (9);
	\draw (9) to (10);
	\draw (10) to (11);
	\draw[dashed] (11) to (13);
	\draw (13) to (14);
	\draw (14) [dashed] -- (15);

	\draw (4) to[bend left] (9);
	\draw (9) to[bend left] (14);
	\draw[dashed] (0.6,0.4) to[bend left] (4);
	\draw[dashed] (14) to[bend left] (15.5,0.6);
%
%
%
%




	\end{tikzpicture}
	\caption{The harmonic labeling of $\langle (0,k)\rangle$ is finitely spanned by $\{0,1\}$ for every $k\in \mathbb{Z}$.}
	\label{Figure:P_B_Remark}
\end{figure}
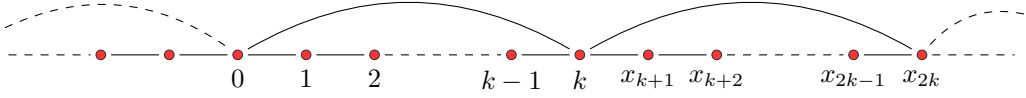
\end{obs}

\section{A Characterization of weak harmonic labelings}


In this section we characterize weakly labeled graphs in terms of certain collection of sets  of integers which we call \emph{harmonic subsets of $\mathbb{Z}$}.

\begin{defi}\label{harmonic} Given a non-empty finite subset $A\subset\mathbb{Z}$ we let $$av(A)=\frac{1}{|A|}\sum\limits_{k\in A}k.$$ Here $|A|$ denotes the cardinality of $A$. We say that $A$ is an \emph{harmonic subset of} $\mathbb{Z}$ if $av(A)\in A$.\end{defi}

\begin{obs}\label{Remark 3 Elements} Note that every unit subset of $\mathbb{Z}$ is harmonic; we call them \textit{trivial harmonic subsets}. Also, there are no two-element harmonic subsets of $\mathbb{Z}$. Therefore, any non-trivial harmonic subset of $\mathbb{Z}$ has at least three elements.\end{obs}


We shall show that certain collections of harmonic subsets of $\mathbb{Z}$ characterize weakly labeled graphs. For this, we consider pairs $(G,\ell)$ of a graph $G$  and a weak harmonic labeling  $\ell$ over $G$. Define an \textit{isomorphism} between two weakly labeled graphs $(G,\ell) $ and $(G',\ell')$ as a graph isomorphism $f:G\rightarrow G'$ such that $\ell'(f(v))=\ell(v)$ for every $v\in V_G$. We let $\mathcal{G}$ denote the quotient set of pairs $(G,\ell)$ under the isomorphism relation.

Given $(G,\ell)\in\mathcal{G}$ we consider the collection $$\mathcal{A}_{(G,\ell)}=\{A_{v}\,:\, v\in V_G\setminus S_G\}$$ where $A_{v}=\{\ell(w)\,:\, w\in N_G(v)\}$. It is easy to see that $\mathcal{A}_{(G,\ell)}$ is a well-defined collection of non-trivial harmonic subsets of $\mathbb{Z}$ such that $av(A_{v})=\ell(v)$. In particular, $A_v\neq A_u$ if $v\neq u$. Also, this collection is finite if and only if $G$ is finite. Furthermore, the collection $\mathcal{A}_{(G,\ell)}$ satisfies the following conditions (whose easy verification are left to the reader).

\begin{lema}\label{Lemma:  4 Properties} Let $\mathcal{A}$ be the collection $\mathcal{A}_{(G,\ell)}$ of harmonic subsets of $\mathbb{Z}$  defined as above. For $A,B\in\mathcal{A}$, we have:\begin{enumerate}
\item[(P1)] $\bigcup_{C\in\mathcal{A}_{(G,\ell)}} C$ is an integer interval. 
\item[(P2)] $av(A)\neq av(B)$ if $A\neq B$
\item[(P3)] If $t\in A\cap B$, $A\neq B$, then there exists $C\in \mathcal{A}$ such that $av(C)=t$.
\item[(P4)] If $av(A)\in B$ then $av(B)\in A$
\item[(P5)] There exists a sequence $A_{i_1},\ldots,A_{i_r}\subset \mathcal{A}$ such that $A_{i_1}=A$, $A_{i_r}=B$ and $av(A_{i_j})\in A_{i_{j+1}}$ for $1\leq j\leq r-1$ (connectedness condition).
\end{enumerate}\end{lema}


Note that  \textit{(P2)} implies that the $t$ in \textit{(P3)} is unique. Actually, \textit{(P2)} is covered by requesting the unicity of $t$ in \textit{(P3)}. However, we state it in this form for computational reasons that will become evident later. On the other hand, \textit{(P5)} is a direct consequence of the connectedness of $G$.


The main result of this section is that properties \textit{(P1)} through \textit{(P5)} of Lemma \ref{Lemma:  4 Properties} characterize weak harmonic labelings, in the sense that $(G,\ell)\mapsto \mathcal{A}_{(G,\ell)}$ is a bijection between $\mathcal{G}$ and the class $\mathcal{H}$ of collections of non-trivial harmonic subsets of $\mathbb{Z}$ satisfying \textit{(P1)} through \textit{(P5)}. Furthermore, if $\mathcal{G}_I\subset\mathcal{G}$ is the subset of pairs $(G,\ell)$ for which $\ell$ is a weak harmonic labeling onto $I$ and $\mathcal{H}_I\subset \mathcal{H}$ is the class of collections $\mathcal{A}$ for which $\bigcup_{C\in\mathcal{A}} C=I$ then the bijection takes $\mathcal{G}_I$ onto $\mathcal{H}_I$.

Note that the map $(G,\ell) \mapsto \mathcal{A}_{(G,\ell)}$ sends elements of $\mathcal{G}_I$ to $\mathcal{H}_I$ by Remark \ref{Obs:dosverticesunonohoja}. We next build the inverse map $\mathcal{H}_I\rightarrow \mathcal{G}_I$. Let $\mathcal{A}=\{A_i\}_{i\in J}\in \mathcal{H}_I$. We define the associated graph $G_{\mathcal{A}}$ as follows:\begin{itemize}
\item $V_{G_{\mathcal{A}}}=I$
\item $i\sim j \Leftrightarrow (\exists\, t / i=av(A_t)\text{ and }j\in A_t)\text{ or }(\exists\, t / j=av(A_t)\text{ and }i\in A_t$).
\end{itemize}
Furthermore, we define a vertex labeling $\ell_{\mathcal{A}}:V_{G_{\mathcal{A}}}\rightarrow I$ by $\ell_{\mathcal{A}}(i)=i$. Lemmas \ref{Lemma Conected} and \ref{Lemma Core} and Corollary \ref{Coro labeling harmonic} below prove that $(G_{\mathcal{A}},\ell_{\mathcal{A}})\in\mathcal{G}_I$.

\begin{lema}\label{Lemma Conected} With the notations as above, $G_{\mathcal{A}}$ is connected.\end{lema}

\begin{proof} Let $p,q\in I$. By \textit{(P1)} there exists $A'_{p},A'_{q}\in\mathcal{A}$ such that $p\in A'_{p}$ and $q\in A'_{q}$. Note that either $p=av(A_{i_p})$ or $p\sim av(A_{i_p})$ (and analogously with $q$). By \textit{(P5)} there exists a sequence $A'_{p}=A_{i_1},A_{i_2},\ldots,A_{i_r}=A'_{q}$ such that $av(A_{i_j})\in A_{i_{j+1}}$ for every $1\leq j\leq r$. In particular, $av(A_{i_j})\sim av(A_{i_{j+1}})$ for every $1\leq j\leq r$. Hence, the walk $p, av(A_{i_1}),av(A_{i_2}),\ldots,av(A_{i_r}),q$ connects $p$ with $q$.\end{proof}

\begin{lema}\label{Lemma Core} With the notations as above, $i\in V_{G_{\mathcal{A}}}\setminus S_{G_{\mathcal{A}}}$ if and only if $\exists\, t\in J$ such that $i=av(A_t)$. Furthermore, this $t$ is unique and $N_{G_{\mathcal{A}}}(i)=A_{t}$.\end{lema}

\begin{proof} If $i\in V_{G_{\mathcal{A}}}\setminus S_{G_{\mathcal{A}}}$ then there exist $j_1\neq j_2$ such that $j_1,j_2\in N_{G_{\mathcal{A}}}(i)$. If $i\neq av(A_t)$ for every $t$ then $\exists\, t_1, t_2$ such that $j_1=av(A_{t_1})$, $j_2=av(A_{t_2})$ and $i\in A_{t_1}\cap A_{t_2}$. But then  \textit{(P3)} implies the existence of $t$ such that $i=av(A_t)$, contradicting our assumption.

Suppose now that $i=av(A_t)$ for some $t$. In particular, $i\in A_{t}$ (because $A_t$ is an harmonic subset). By Remark \ref{Remark 3 Elements}, there exist $j_1,j_2\in A_{t}$ non-equal such that $j_1,j_2\neq av(A_t)$. Thus $j_1,j_2\in N_{G_{\mathcal{A}}}(i)$  by the definition of adjacency in $G_{\mathcal{A}}$ and $i\in V_{G_{\mathcal{A}}}\setminus S_{G_{\mathcal{A}}}$.

The uniqueness of $t$ is a direct consequence of \textit{(P2)}. Now, if $j\in N_{G_{\mathcal{A}}}(i)$ then either ($\exists\, s / i=av(A_s)\text{ and }j\in A_s)\text{ or }(\exists\, s / j=av(A_s)\text{ and }i\in A_s$). In the first case $s=t$ by unicity. In the latter case,  \textit{(P4)} implies that $j=av(A_s)\in A_{t}$. In any case $j\in A_{t}$, which proves $N_{G_{\mathcal{A}}}(i)\subset A_{t}$. Now, if $j\in A_{t}$ then $i\sim j$ by the definition of adjacency of $G_{\mathcal{A}}$. Hence, $j\in N_{G_{\mathcal{A}}}(i)$.\end{proof}

\begin{coro}\label{Coro labeling harmonic} With the notations as above, $\ell_{\mathcal{A}}$ is a weak harmonic labeling over $G_{\mathcal{A}}$.\end{coro}

\begin{proof} If $i\in V_{\mathcal{A}}\setminus S_{\mathcal{A}}$, let $t$ be such that $i=av(A_t)$. Then $$\ell_{\mathcal{A}}(i)=i=av(A_t)=\frac{1}{|A_{t}|}\sum_{k\in A_{t}} k=\frac{1}{|N_{G_{\mathcal{A}}}(i)|}\sum_{k\in N_{G_{\mathcal{A}}}(i)} k=\frac{1}{\deg(i)+1}\sum_{\substack{k\sim i \\ k=i}} \ell_{\mathcal{A}}(k).$$\end{proof}

\begin{teo}\label{Theorem:MAIN} The maps $(G,\ell)\mapsto \mathcal{A}_{(G,\ell)}$ and $\mathcal{A}\mapsto (G_{\mathcal{A}},\ell_{\mathcal{A}})$ are mutually inverse.\end{teo}

\begin{proof}
Define the function $f:(G,\ell)\rightarrow (G_{\mathcal{A}_{(G,\ell)}},\ell_{\mathcal{A}_{(G,\ell)}})$ as $f(v)=\ell(v)$. We will show that $f$ is a graph isomorphism between $G$ and $G_{\mathcal{A}_{(G,\ell)}}$ and that $\ell_{\mathcal{A}_{(G,\ell)}}(f(v))=\ell(v)$. Since $\ell$ is a weak harmonic labeling then $f$ is a bijection between $V_G$ and $I$, so it suffices to show that $v\sim w$ if and only if $f(v)\sim f(w)$. Now, if $v\sim w$ then either $v$ or $w$ must belong to the set of non-leaves of $G$ (Remark \ref{Obs:dosverticesunonohoja}). Assume $v\in V_G\setminus S_G$. Then, by  definition of $\mathcal{A}_{(G,\ell)}$ it exists $A_v$ with $av(A_v)=\ell(v)$. Also, since $v\sim w$ then $w\in N_G(v)$ and hence $\ell(w)\in A_v$. Therefore $\ell(v)\sim \ell(w)$; that is, $f(v)\sim f(w)$.

Now, suppose $f(v)\sim f(w)$. Then $\ell(v)\sim \ell(w)$ in $G_{\mathcal{A}_{(G,\ell)}}$. Then, either ($\exists u\in V_G\setminus S_G$ such that $ \ell(v)=av(A_{u})$ and $\ell(w)\in A_{u}$) or ($\exists x\in V_G\setminus S_G$ such that $\ell(w)=av(A_{x})$ and $\ell(v)\in A_{x}$). Without loss of generality we may assume the first case happens. Since $\ell$ is a bijection then $w$ must belong to $N_G(v)$. Hence $w\sim v$. This proves that $G$ is isomorphic to $G_{\mathcal{A}_{(G,\ell)}}$.

Finally, from the definition of $\ell_{\mathcal{A}_{(G,\ell)}}$:
$$\ell_{\mathcal{A}_{(G,\ell)}}(f(v))=f(v)=\ell(v),$$ which finishes proving that $(G,\ell)\mapsto \mathcal{A}_{(G,\ell)}\mapsto (G_{\mathcal{A}_{(G,\ell)}},\ell_{\mathcal{A}_{(G,\ell)}})$ is the identity.

We now  prove that $\mathcal{A}\mapsto (G_{\mathcal{A}},\ell_{\mathcal{A}})\mapsto \mathcal{A}_{(G_{\mathcal{A}},\ell_{\mathcal{A}})}$ is the identity. 
Define $g:\mathcal{A}\rightarrow \mathcal{A}_{(G_{\mathcal{A}},\ell_{\mathcal{A}})}$ as follows: $g(A_t)=\tilde{A}_i$  where $i\in V_{G_{\mathcal{A}}}\setminus S_{G_{\mathcal{A}}}$ is such that $i=av(A_t)$ (Lemma \ref{Lemma Core}). Note that $g$ is one to one by (\textit{P2}) and the fact that $i\neq j$ implies $\tilde{A}_i\neq \tilde{A}_j$ in $\mathcal{A}_{(G_{\mathcal{A}},\ell_{\mathcal{A}})}$ (see properties of $\mathcal{A}_{(G,\ell)}$ before Lemma \ref{Lemma:  4 Properties}). Also, Lemma \ref{Lemma Core} implies that $g$ is onto. Since $\ell_{\mathcal{A}}(s)=s$ and $A_t=N_{G_{\mathcal{A}}}(i)$ (again by Lemma \ref{Lemma Core}) then $\tilde{A}_i=\{\ell_{\mathcal{A}}(s)\,:\, s\in N_{G_{\mathcal{A}}}(i)\}=N_{G_{\mathcal{A}}}(i)=A_t$. 

\end{proof}

Theorem \ref{Theorem:MAIN} provides a concrete way to compute weak harmonic labelings of finite graphs. A list of all possible weakly labeled graphs up to ten vertices can be found in the Appendix.

\subsection*{On non-connected graphs.} 
In \cite{BCPT}, harmonic labelings are defined for general graphs (not necessarily connected ones). However, the non-connected case gives rise to many superfluous examples, as the following construction shows. Given a graph $G$ and an harmonic labeling $\ell:V_G\rightarrow\mathbb{Z}$, let $H=\bigvee_{1\leq i\leq k} G_i$ be the disjoint union of $k\in\mathbb{Z}$ copies of $G$. Then, we can define an harmonic labeling $\ell_H$ over $H$ as follows: $$\ell_H(v)=k\ell(v)+i-1\text{, if $v\in V_{G_i}$}.$$

The definitions and results for the connected case can be extended to the non-connected case in a straightforward manner as long as every connected components of $G$ have at least three vertices. The case for connected components with less than three vertices give rise to uninteresting examples as these components are ``invisible" to the requirement of harmonicity and can be used to complete partial one to one labelings. Even with these requirements,  harmonically labeled non-connected graphs are in great amount uninteresting examples, which arise from simply disconnecting connected cases (see Figure \ref{Figure:DisconnectedOne} (Top)). The first non-trivial examples appear on 8-vertex graphs and are shown in Figure \ref{Figure:DisconnectedOne} (Bottom).

\medskip

\begin{figure}
\begin{minipage}{0.49\textwidth}
\begin{center}

\tiny

	\begin{tikzpicture}[scale=0.6]

	\node [draw,circle](0)at(0,0){$0$};
	\node [draw,circle](1)at(2,0){$1$};
	\node [draw,circle](2)at(4,0){$2$};
	
	\node [draw,circle](3)at(0,2){$3$};
	\node [draw,circle](4)at(2,2){$4$};
	\node [draw,circle](5)at(4,2){$5$};
	\draw (0)--(1);
	\draw (1)--(2);
	
	\draw (3)--(4);
	\draw (4)--(5);
	\end{tikzpicture}\\
	\vspace{5mm}
	$ \mathcal{A}=\lbrace 012; 345\rbrace$	
\end{center}
\end{minipage}\begin{minipage}{0.49\textwidth}
\begin{center}

\tiny

	\begin{tikzpicture}[scale=0.6]

	\node [draw,circle](0)at(0,0){$1$};
	\node [draw,circle](1)at(2,0){$3$};
	\node [draw,circle](2)at(4,0){$5$};
	
	\node [draw,circle](3)at(0,2){$0$};
	\node [draw,circle](4)at(2,2){$2$};
	\node [draw,circle](5)at(4,2){$4$};
	\draw (0)--(1);
	\draw (1)--(2);
	
	\draw (3)--(4);
	\draw (4)--(5);
	\end{tikzpicture}\\
	\vspace{5mm}
	$ \mathcal{A}=\lbrace 135; 024\rbrace$	
\end{center}
\end{minipage}

\vspace{0.5in}

\begin{minipage}{0.49\textwidth}
\begin{center}
\tiny	
	
	\begin{tikzpicture}[scale=0.6]

	\node [draw,circle](2) at(0,1.5){$4$};
	\node [draw,circle](0)at(0,4){$1$};	
	\node [draw,circle](3)at(-1.8,-0.8){$5$};
	\node [draw,circle](7)at(1.8,-0.8){$6$};

	\node [draw,circle](5)at(4,2){$3$};
	
	\node [draw,circle](8)at(2,4){$0$};
	\node [draw,circle](1)at(6,4){$2$};
	\node [draw,circle](6)at(4,-0.8){$7$};

	\draw (1) to (5);
	\draw (8) to (5);
	\draw (6) to (5);

	\draw (2) to (0);
    \draw (2) to (3);   	
	\draw (7) to (2);

	\end{tikzpicture}\\
	\vspace{5mm}	
	$ \mathcal{A}=\lbrace 1456; 0237 \rbrace$
\end{center}
\end{minipage}\begin{minipage}{0.49\textwidth}
\begin{center}

\tiny

	\begin{tikzpicture}[scale=0.6]

	\node [draw,circle](2) at(0,1.5){$4$};
	\node [draw,circle](0)at(0,4){$0$};	
	\node [draw,circle](3)at(-1.8,-0.8){$5$};
	\node [draw,circle](7)at(1.8,-0.8){$7$};

	\node [draw,circle](5)at(4,2){$3$};
	
	\node [draw,circle](8)at(2,4){$1$};
	\node [draw,circle](1)at(6,4){$2$};
	\node [draw,circle](6)at(4,-0.8){$6$};

	\draw (1) to (5);
	\draw (8) to (5);
	\draw (6) to (5);

	\draw (2) to (0);
    \draw (2) to (3);   	
	\draw (7) to (2);

	\end{tikzpicture}\\
	\vspace{5mm}	
	$ \mathcal{A}=\lbrace 0457; 1236 \rbrace$
\end{center}

\end{minipage}

\caption{\textbf{Top.} Trivial examples of disconnected weakly labeled graphs. \textbf{Bottom.} Non-trivial examples of disconnected weakly labeled graphs.}
\label{Figure:DisconnectedOne}
\end{figure}
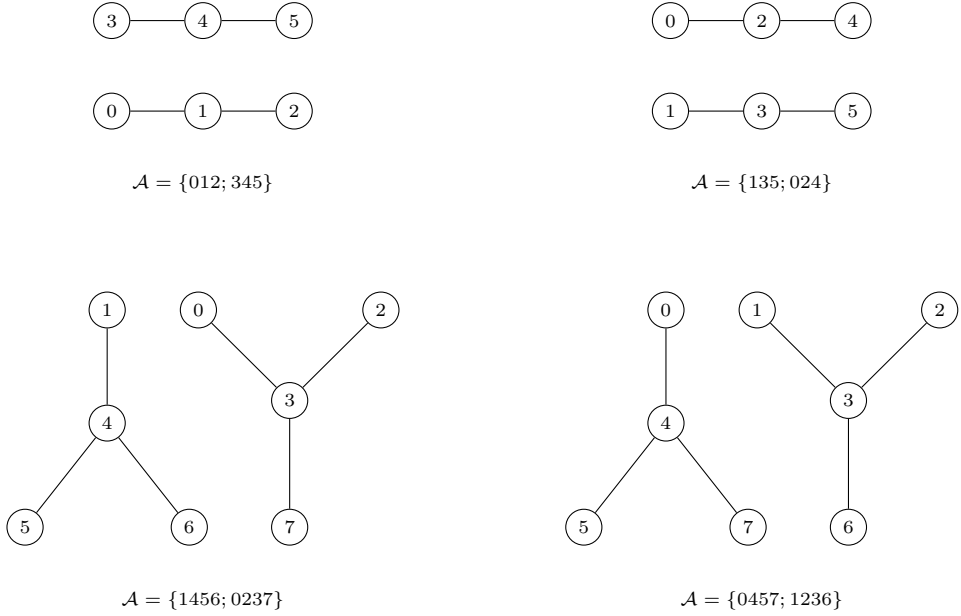

\normalsize

%
%
%
%
%
%
%
%
%
%
%
%
%
%
%
%
%
%
%
%
%
%
%
%
%
%
%
%
%
%
%
%
%

The same characterization given in Theorem \ref{Theorem:MAIN} also holds for non-connected graphs provided that the condition (\textit{P5}) is dropped from Lemma \ref{Lemma:  4 Properties}. Actually, it is straightforward to see that $(G,\ell)$ (resp. $G_{\mathcal{A}}$) is connected if and only if $\mathcal{A}_{(G,\ell)}$ (resp. $\mathcal{A}$) satisfies $(P5)$. Particularly, if we let $\tilde{\mathcal{G}}_{\mathbb{Z}}\subset \mathcal{G}_{\mathbb{Z}}$ denote the set of pairs $(G,\ell)$ for which $S_G=\emptyset$ ($G$ not necessarily connected) then $\tilde{\mathcal{G}}_{\mathbb{Z}}$ is the set of harmonically labeled graphs as defined in \cite{BCPT}. From the above considerations, we obtain the following characterization of harmonic labelings.

\begin{teo}\label{Theorem:MAINMAIN} A (non-necessarily connected) graph $G$ admits an harmonic labeling $\ell$ if and only if $S_G=\emptyset$ and $\mathcal{A}_{(G,\ell)}$ satisfies:
\begin{enumerate}
\item[(ZP1)] For every $k\in \mathbb{Z}$ there exists $C\in \mathcal{A}_{(G,\ell)}$ such that $av(C)=k$.
\item[(ZP2)] $av(A)\neq av(B)$ if $A\neq B$.
\item[(ZP3)] If $av(A)\in B$ then $av(B)\in A$.
\end{enumerate}\end{teo}

\begin{proof} The result follows from Theorem \ref{Theorem:MAIN} by noting that \emph{P1} transforms into \emph{ZP1} and that \emph{P3} is covered by \emph{ZP1}.\end{proof}


\section{Multigraphs and total labelings}

In this section we extend the main definitions and results of weak harmonic labelings to multigraphs and provide a generalization of Theorem \ref{Theorem:MAIN} in this context. \textit{All multigraphs are connected, loopless and have bounded degree (see Remark \ref{generalizacion a loops}).} Also, since the identity of the edges is indifferent to the theory, we consider all parallel edges to be indistinguishable.

Recall that a (finite) \textit{multiset} $\mathcal{M}$ is a pair $(A,m)$ where $A$ is a (finite) non-empty set and $m:A\rightarrow \mathbb{N}$ is a function giving the \textit{multiplicity} of each element in $A$ (the number of instances of that element). The cardinality of $\mathcal{M}$ is the number $\vert \mathcal{M}\vert=\sum_{x\in A} m(x)$. If $A=\lbrace x_1,x_2,\cdots x_n\rbrace$ we shall often write $\mathcal{M}=\lbrace x_1^{m(x_1)},x_2^{m(x_2)},\ldots, x_n^{m(x_n)} \rbrace$. If $m(x_i)=1$ we simply write $x_i$.


Given a multigraph $G$ we let $m_G(v,w)=m_G(w,v)\in \mathbb{Z}_{\geq 0}$ denote the number of edges between vertices $v,w\in V_G$, $v\neq w$. If $m_G(v,w)\neq 0$ then $v$ and $w$ are adjacent and we write $v\sim w$. If $m(v,w)=k\geq 2$ we shall often write $v\stackrel{k}{\sim} w$ or $\{v,w\}^k\in G$. A vertex $v\in G$ is a \textit{leaf} if $m_G(v,w)\neq 0$ for exactly one $w\neq v$. As in the simple case, we shall denote $S_G$ the set of leaves of the multigraph $G$.
%
%
%

The \textit{simplification} of a multigraph $G$ is the simple graph $sG$ where $V_{sG}=V_G$ and $\{u,v\}\in E_{sG}$ if and only if $m_G(v,w)\neq 0$ ($v\neq w$). We shall call the \textit{closed multi neighborhood} of $v\in V_G$ in a multigraph $G$ to the multiset $\mathcal{N}_G(v)=\{v\}\cup\{w^{m_G(v,w)}\,:\, v\sim w\}$. Thus, the closed multi neighborhood of $v$ keeps track of the multiplicities of the vertices adjacent to $v$ as well. The (standard) close neighborhood of $v$ is $N_{sG}(v)\subset V_G.$


\begin{defi}\label{Def:HarmonicMulti} A \emph{weak harmonic labeling} of a multigraph $G$ is a bijective function $\ell: V_G\rightarrow I$ such that 
$$
\ell(v)=\frac{1}{deg(v)}\sum_{w\sim v}m_G(v,w)\phi(w)\hspace{0.2in}\forall v \in V_G\setminus S_G.
$$\end{defi}

Figure \ref{Figure:MultilabelingONE} shows some examples of harmonic labelings of finite multigraphs. Note that the presence of at least two leaves is still a requirement for the existence of a weak harmonic labeling.

\footnotesize

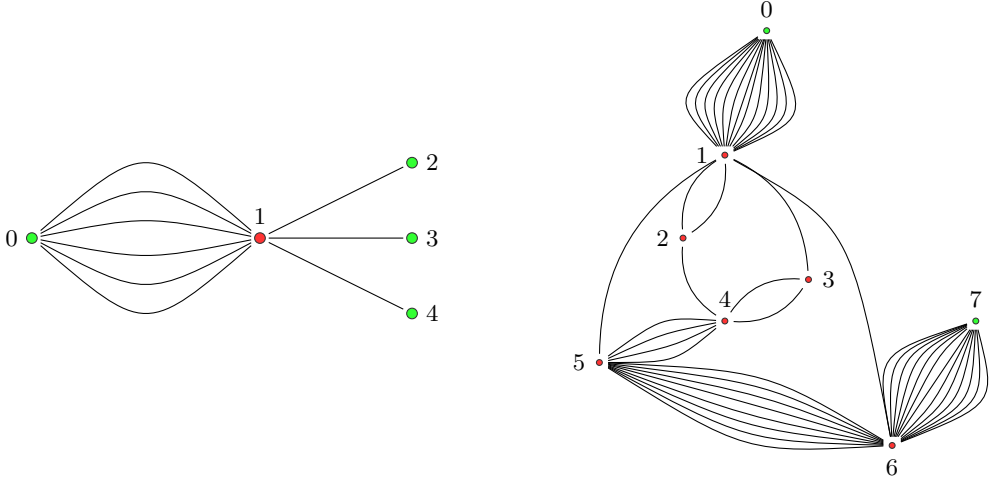
\begin{figure}
\begin{minipage}{0.49\textwidth}
\begin{center}

\begin{tikzpicture}[scale=1]
	\node(0)at(0,0){};
	\filldraw[fill=green!80,draw=black!80] (0) circle (2pt)
	node[left=2pt] {$0$}; 
	\node(1)at(3,0){};
	\filldraw[fill=red!80,draw=black!80] (1) circle (2pt)
	node[above=2pt] {$1$}; 
	\node(2)at(5,1){};
	\filldraw[fill=green!80,draw=black!80] (2) circle (2pt)
	node[right=2pt] {$2$}; 
	\node(3)at(5,0){};
	\filldraw[fill=green!80,draw=black!80] (3) circle (2pt)
	node[right=2pt] {$3$}; 
	\node(4)at(5,-1){};
	\filldraw[fill=green!80,draw=black!80] (4) circle (2pt)
	node[right=2pt] {$4$}; 


	\draw (1)to  (2);
	\draw (1)to  (3);
	\draw (1)to  (4);

	\draw (0) .. controls (1.5,1.3) .. (1);
	\draw (0) .. controls (1.5,0.8) .. (1);
	\draw (0) .. controls (1.5,0.3) .. (1);
	\draw (0) .. controls (1.5,-0.3) .. (1);
	\draw (0) .. controls (1.5,-0.8) .. (1);
	\draw (0) .. controls (1.5,-1.3) .. (1);
	
\end{tikzpicture}
\end{center}
\end{minipage}\begin{minipage}{0.49\textwidth}
\begin{center}

\begin{tikzpicture}[scale=.55]
	\node(-1)at(2,4){};
	\filldraw[fill=green!80,draw=black!80] (-1) circle (2pt)
	node[above=2pt] {$0$}; 
	\node(0)at(1,1){};
	\filldraw[fill=red!80,draw=black!80] (0) circle (2pt)
	node[left=3pt] {$1$}; 
	\node(1)at(0,-1){};
	\filldraw[fill=red!80,draw=black!80] (1) circle (2pt)
	node[left=2pt] {$2$}; 
	\node(2)at(3,-2){};
	\filldraw[fill=red!80,draw=black!80] (2) circle (2pt)
	node[right=2pt] {$3$}; 
	\node(3)at(1,-3){};
	\filldraw[fill=red!80,draw=black!80] (3) circle (2pt)
	node[above=2pt] {$4$}; 
	\node(4)at(-2,-4){};
	\filldraw[fill=red!80,draw=black!80] (4) circle (2pt)
	node[left=2pt] {$5$}; 
	\node(5)at(5,-6){};
	\filldraw[fill=red!80,draw=black!80] (5) circle (2pt)
	node[below=2pt] {$6$}; 
	\node(6)at(7,-3){};
	\filldraw[fill=green!80,draw=black!80] (6) circle (2pt)
	node[above=2pt] {$7$}; 

	
	\draw (0)to [bend right] (1);
	\draw (0)to [bend left] (1);
	\draw (0)to [bend left] (2);
	\draw (0)to [bend right] (4);

	\draw (-1)to (0);
	\draw (-1) .. controls (1.2,2.4) .. (0);
	\draw (-1) .. controls (0.9,2.4) .. (0);
	\draw (-1) .. controls (0.6,2.4) .. (0);
	\draw (-1) .. controls (0.3,2.4) .. (0);
	\draw (-1) .. controls (0.0,2.4) .. (0);
	\draw (-1) .. controls (-0.3,2.4) .. (0);
	\draw (-1) .. controls (3,1.9) .. (0);
	\draw (-1) .. controls (2.7,1.9) .. (0);
	\draw (-1) .. controls (2.4,1.9) .. (0);
	\draw (-1) .. controls (2.1,1.9) .. (0);
	\draw (-1) .. controls (1.8,1.9) .. (0);	
	\draw (-1) .. controls (1.5,1.9) .. (0);
	
	\draw (0).. controls (4,-0.7) .. (5);
	
	\draw (2)to [bend left] (3);
	\draw (2)to [bend right] (3);
	\draw (1)to [bend right] (3);
	
	\draw (3) .. controls (-0.5,-2.9) .. (4);
	\draw (3) .. controls (-0.5,-3.3) .. (4);
	\draw (3) .. controls (-0.2,-3.6) .. (4);
	\draw (3) .. controls (-0.2,-4) .. (4);
	
	\draw (4) .. controls (2,-4.2) .. (5);
	\draw (4) .. controls (2,-4.5) .. (5);
	\draw (4) .. controls (2,-4.8) .. (5);
	\draw (4) .. controls (2,-5.1) .. (5);
	\draw (4) .. controls (1.5,-5.3) .. (5);	
	\draw (4) .. controls (1.5,-5.6) .. (5);
	\draw (4) .. controls (1.5,-5.9) .. (5);
	\draw (4) .. controls (1.5,-6.2) .. (5);
	
	\draw (5) .. controls (4.6,-3.8) .. (6);
	\draw (5) .. controls (4.9,-3.8) .. (6);
	\draw (5) .. controls (5.2,-3.8) .. (6);
	\draw (5) .. controls (5.5,-3.8) .. (6);
	\draw (5) .. controls (5.8,-3.8) .. (6);
	\draw (5) .. controls (6.1,-3.8) .. (6);
	\draw (5)to (6);
	\draw (5) .. controls (6,-5) .. (6);
	\draw (5) .. controls (6.3,-5) .. (6);
	\draw (5) .. controls (6.6,-5) .. (6);
	\draw (5) .. controls (6.9,-5) .. (6);
	\draw (5) .. controls (7.2,-5) .. (6);
	\draw (5) .. controls (7.5,-5) .. (6);

\end{tikzpicture}
\end{center}
\end{minipage}
\caption{Examples of weak harmonic labeling on multigraphs}
\label{Figure:MultilabelingONE}
\end{figure}
\normalsize

We next show that Theorem \ref{Theorem:MAIN} can be generalized to multigraphs.

\begin{defi} For a multiset $\mathcal{M}=(A,m)$ with finite non-empty $A\subset \mathbb{Z}$ we let $$av(\mathcal{M})=\frac{1}{|\mathcal{M}|}\sum\limits_{k\in A}m(k)k.$$ We say that $\mathcal{M}$ is an \emph{harmonic multiset of $\mathbb{Z}$} if $av(\mathcal{M})\in A$.\end{defi}

\begin{obs}\label{Remark 3 Elements Multigraphs} As for harmonic subsets, the multisets whose underlying set is a unit set of $\mathbb{Z}$ are (trivial) harmonic multisets of $\mathbb{Z}$. Also, there are no harmonic multisets of $\mathbb{Z}$ whose underlying set has two elements. Therefore, any non-trivial harmonic multiset of $\mathbb{Z}$ has an underlying set of at least three elements.\end{obs}

%
%
%
%
%
%

Analgously to the simple case, we consider pairs $(G,\ell)$ for a multigraph $G$ and a weak harmonic labeling $\ell:V_G\rightarrow I$ and define an \textit{isomorphism} between two weakly labeled multigraphs $(G,\ell)$ and $(G',\ell')$ as a multigraph isomorphism $f:G\rightarrow G'$ such that $\ell(f(v))=\ell(v)$ for every $v\in V_G$.  We let $\mathcal{MG}_I$ denote the quotient set of pairs $(G,\ell)$, $\ell:V_G\rightarrow I$, under the isomorphism relation.

Given $(G,\ell)\in\mathcal{MG}_I$ we consider the collection $$\mathcal{MA}_{(G,\ell)}=\{\mathcal{B}_{v}\,:\, v\in V_G\setminus S_G\}$$ where $\mathcal{B}_{v}=\{\ell(v)\}\cup \{\ell(w)^{m_G(v,w)}\,:\, w\sim v\}$. As in the simple case, it is easy to see that $\mathcal{MA}_{(G,\ell)}$ is a collection of non-trivial harmonic multisets of $\mathbb{Z}$ verifying  $av(\mathcal{B}_{v})=\ell(v)$ that satisfies the (analoguous) conditions than Lemma \ref{Lemma:  4 Properties}. Namely, if  $A_{\mathcal{M}}$ stands for the underlying set of the multiset $\mathcal{M}$:

\begin{lema}\label{Lemma:  4 Properties Multigraphs} 

Let $\mathcal{MA}$ be the collection $\mathcal{MA}_{(G,\ell)}$ of harmonic multisets of $\mathbb{Z}$  defined as above. For $\mathcal{B},\mathcal{C}\in\mathcal{MA}$, we have:

\begin{enumerate}
\item[(MP1)] $\bigcup_{\mathcal{D}\in\mathcal{MA}} A_{\mathcal{D}}=I$.
\item[(MP2)] $av(\mathcal{B})\neq av(\mathcal{C})$ if $\mathcal{B}\neq \mathcal{C}$.
\item[(MP3)] If $t\in A_{\mathcal{B}}\cap A_{\mathcal{C}}$ then there exists $\mathcal{D}\in \mathcal{MA}$ such that $av(\mathcal{D})=t$.
\item[(MP4)] If $av(\mathcal{B})^k\in \mathcal{C}$ then $av(\mathcal{C})^k\in \mathcal{B}$.
\item[(MP5)] There exists a sequence $\mathcal{B}_{i_1},\ldots,\mathcal{B}_{i_r}\subset\mathcal{MA}$ such that $\mathcal{B}_{i_1}=\mathcal{B}$, $\mathcal{B}_{i_r}=\mathcal{C}$ and $av(\mathcal{B}_{i_j})\in \mathcal{B}_{i_{j+1}}$ for $1\leq j\leq r-1$ (connectedness condition).
\end{enumerate}\end{lema}

We let $\mathcal{MH}_I$ stand for the class of collections of non-trivial harmonic multisets of $\mathbb{Z}$ with $\bigcup_{\mathcal{D}\in\mathcal{MA}} A_{\mathcal{D}}=I$ satisfying \textit{(MP1)} through \textit{(MP5)} of Lemma \ref{Lemma:  4 Properties Multigraphs}. With the analogous constructions as in the simple case it can be shown that there is a bijection $\mathcal{MG}_I\equiv\mathcal{MH}_I$.
%
%
%
Namely, for $\mathcal{MA}=\{\mathcal{B}_i\}_{i\in J}\in \mathcal{MH}_I$ define the associated multigraph $G_{\mathcal{MA}}$ as:
\begin{itemize}
\item $V_{\mathcal{MA}}=I$
\item $i\stackrel{k}{\sim}j\in G_{\mathcal{MA}} \Leftrightarrow  (\exists\, t / i=av(\mathcal{B}_t)\text{ and }j^k\in \mathcal{B}_t)\text{ or }(\exists\, t / j=av(\mathcal{B}_t)\text{ and }i^k\in \mathcal{B}_t) $         
\end{itemize}
Note that, by \textit{(MP4)}, this multigraph is well-defined. Finally, we define a vertex labeling $\ell_{\mathcal{MA}}$ over $G_{\mathcal{MA}}$ by $\ell_{\mathcal{MA}}(i)=i$.



Identical arguments as in the proofs of Lemmas \ref{Lemma Conected} and \ref{Lemma Core}, Corollary \ref{Coro labeling harmonic} and Theorem \ref{Theorem:MAIN} go through to prove the following analogous results for multigraphs.



\begin{lema} With the notations as above,
\begin{enumerate}
\item $G_{\mathcal{MA}}$ is connected.
\item $i\in V_{G_{\mathcal{MA}}}\setminus S_{G_{\mathcal{MA}}}$ if and only if $\exists\,t\in J$ such that $i=av(\mathcal{B}_t)$. Furthermore, this $t$ is unique and $\mathcal{N}_{G_{\mathcal{MA}}}(i)=\mathcal{B}_t$. In particular, $j\in A_{\mathcal{B}_t}$ if and only if $j=i$ or $j\sim i$ in $G_{\mathcal{MA}}$.
\item $\ell_{\mathcal{MA}}$ is a weak harmonic labeling over $G_{\mathcal{MA}}$.
\end{enumerate}\end{lema}

\begin{teo}\label{Theorem:MAIN Multigraphs} The maps $(G,\ell)\rightarrow \mathcal{MA}_{(G,\ell)}$ and $\mathcal{MA}\rightarrow (G_{\mathcal{MA}},\ell_{\mathcal{MA}})$ are mutually inverse.\end{teo}

\begin{obs}\label{generalizacion a loops} All the results of this section can be extended in a straighforward manner to multigraphs with loops.  This is consequence of the fact that a multiset $$\{ x_1^{m_1},x_2^{m_{2}},\dots, x_k,\ldots, x_n^{m_{n}}\}$$ is a harmonic with average $x_k$ if and only if $\{ x_1^{m_1},x_2^{m_{2}},\dots, x_k^{m},\ldots, x_n^{m_{n}}\}$ is harmonic with average $x_k$ for all $k>0$.\end{obs}

%
%

\subsection*{Total weak harmonic labelings} Since a weak harmonic labeling over a multigraph $G$ is trivially equivalent to a total labeling over $sG$ we can state the theory in terms of total labelings.

\begin{defi} If $G$ is a simple graph, then we call a \textit{total weak harmonic labeling} of $G$ onto $I$ to a function $\ell: V_G\cup E_G\rightarrow \mathbb{Z}$ such that $\ell |_{V_G}$ is a bijection with $I$ and
$$\ell(v)=\frac{1}{deg(v)}\sum_{w\sim v}\ell(\{v,w\})\ell(w)\hspace{0.2in}\forall v \in V_G\setminus S_G.$$\end{defi}

Note that total weak harmonic labelings have no restriction on the edges. Now, given a weak harmonic labeling $\ell:V\rightarrow I$ over a multigraph $G$ we have the associated total weak harmonic labeling $\ell^*: V_{sG}\cup E_{sG}\rightarrow \mathbb{Z}$ over $sG$ defined as $$\begin{cases}\ell^*(v)=\ell(v)& v\in V_{sG}\\ \ell^*(\{v,w\})=m_G(v,w)&\{u,v\}\in E_{sG}.\end{cases}$$ Conversely, given a total weak harmonic labeling $\ell: V_G\cup E_G\rightarrow \mathbb{Z}$ over a simple graph $G$ then we can define a weak harmonic labeling over the multigraph $G_{\ell}$ where $V_{G_{\ell}}=V_G$ and $m_{G_{\ell}}(v,w)=\ell(\{v,w\})$. View in this fashion, weak harmonic labelings of simple graphs are a particular case of total weak harmonic labelings of simple graphs.


Total weak harmonicity  is naturally much less restrictive than weak harmonicity. Any finite simple graph $G$ which admits a weak harmonic labeling in particular admits a bijective vertex-labeling $\ell:V_G\rightarrow [0,n-1]$ such that \begin{equation}\label{eq:final}\min_{w\in N_v(G)}\{\ell(w)\}<\ell(v)<\max_{u\in N_v(G)}\{\ell(u)\}\end{equation} for every $v\in V_G\setminus S_G$. Algorithm \ref{algovalues} produces a total weak harmonic labeling from any labeling $\phi$ fulfilling \eqref{eq:final} on a finite simple graph $G$. It makes use of the following

\begin{obs}\label{Obs:MultiHarmonize} If $\mathcal{M}=(A,m)$ is a finite multiset and $x\in\mathcal{M}$ is neither the maximum or minimum of $\mathcal{M}$ then we can correct the multiplicities of the elements of $\mathcal{M}$ so $av(\mathcal{M})=x$. Indeed, if $x>av(\mathcal{M})$ then letting $s=\min_{y\in \mathcal{M}}\{y\}$ and $$m'(y)=\begin{cases}m(y)\cdot m(s)\cdot (x-s)& y\neq s, x \\m(s)\cdot|\sum\limits_{z\neq s}(x-z)m(z)|&y=s\end{cases}$$ we readily see that $\mathcal{M}'=(A,m')$ is an harmonic multiset of $\mathbb{Z}$. The case $x<av(\mathcal{M})$ is analogous.

Additionally, note that multiplying the multiplicities of every element in an harmonic multiset of $\mathbb{Z}$ by a fixed positive integer does not alter its harmonicity.\end{obs}

\begin{algorithm}
\caption{\emph{Total weak harmonic labeling}}\label{algovalues}
\begin{algorithmic}[1]
\Require{$\phi:V_G\rightarrow [0,n-1]$ with property \eqref{eq:final}}
\Ensure{$\ell$ a total harmonic labeling on $G$}
\Procedure{totalLabelingFrom}{$\phi$}
\State Order $V_G\setminus S_G=\{v_1,\ldots,v_t\}$ such that $\phi(v_i)<\phi(v_j)$ if $i<j$.
\State $\mathcal{B}_{i} \gets \{\phi(w)\,|\, w\in N_G(v_i)\}$ ($1\leq i\leq t$).

\For{$1\leq i\leq t$}

\If{$av(\mathcal{B}_i)\neq\phi(v_i)$}

\State \multiline{%
Harmonize $\mathcal{B}_i$ by conveniently altering the multiplicity of the elements different from $\phi(v_i)$ (see Remark \ref{Obs:MultiHarmonize}).}
\If{$\phi(v_j)\in \mathcal{B}_i$}
\State \multiline{%
\textbf{For $1\leq j<i$:} Correct the multiplicities of the elements of $\mathcal{B}_j$ so the multiplicity of $\phi(v_i)\in\mathcal{B}_j$ coincides with that of $\phi(v_j)\in\mathcal{B}_i$}
\State  \multiline{%
\textbf{For $i< j\leq t$:} Correct the multiplicity of $\phi(v_i)\in\mathcal{B}_j$ so it coincides with that of $\phi(v_j)\in\mathcal{B}_i$}
\EndIf
\EndIf
%
%
%
%
%
\EndFor
\State $\ell(v)\gets\phi(v)$ for every $v\in V_G$
\State $\ell(\{w,u\})\gets$ multiplicity of $\phi(u)$ in $\mathcal{B}_{\phi(w)}$
\EndProcedure
\end{algorithmic}
\end{algorithm}

Figure \ref{Figure:MultilabelingTWO} shows examples of total weak harmonic labelings obtained from Algorithm \ref{algovalues} to some complete graphs with two leaves added.

\begin{figure}
\begin{minipage}{0.33\textwidth}
\tiny

	\begin{tikzpicture}[scale=0.6]
	
	\node [draw,circle](0)at(1,-2){$0$};
	\node [draw,circle](1) at(2,0){$1$};
	\node [draw,circle](2)at(3,2){$2$};	
	\node [draw,circle](3)at(4,0){$3$};
	\node [draw,circle](4)at(5,-2){$4$};
	\node [above] at (1,-1.35) {$(3)$};
	\node [above] at (5,-1.35) {$(3)$};

	\draw (0) to (1);
	\draw (1) to (2);
	\draw (1) to (3);
	\draw (2) to (3);   	
	\draw (3) to (4);

	\end{tikzpicture}
%
%
%
%
%
%
%
%
\end{minipage}	\begin{minipage}{0.33\textwidth}
\tiny
	\begin{tikzpicture}[scale=0.6]
	
\draw[->, dashed]  (3.3,1.3) to [bend right] (2.75,2.7);	
	
	\node [draw,circle](0)at(1,-2){$0$};
	\node [draw,circle](1) at(2,0){$1$};
	\node [draw,circle](2)at(2,2){$2$};	
	\node [draw,circle](3)at(4,2){$3$};
	\node [draw,circle](4)at(4,0){$4$};
	\node [draw,circle](5)at(5,-2){$5$};
	
	\node [below] at (0.8,-0.5) {$(18)$};
	\node [below] at (5.2,-0.5) {$(18)$};
	\node [left]  at (2,1) { $(3)$};
	\node [right] at (4,1) { $(7)$};
	\node [below] at (3,0) {$(3)$};
	\node [anchor=south east] at (2.8,2.5){$(3)$};
	\draw (0) to (1);
	\draw (1) to (2);
	\draw (1) to (3);
	\draw (1) to (4);
	\draw (2) to (3);   	
	\draw (2) to (4);
	\draw (3) to (4);
	\draw (4) to (5);

	\end{tikzpicture}
%
%
%
%
\end{minipage}\begin{minipage}{0.33\textwidth}
\tiny
\centering
	\begin{tikzpicture}[scale=0.6]
	
	\draw[->, dashed]  (0.52,2.2) to [bend right] (-1,3.2);
	\draw[->, dashed]  (1.4,1.1) to [bend right] (-0.18,3.6);
	\draw[->, dashed]  (1.4,2.4) to [bend left] (2.8,3.6);
	
	\node [draw,circle](0)at(-2.2,-1){$0$};
	\node [draw,circle](1) at(-0.2,0){$1$};
	\node [draw,circle](2)at(-0.7,2){$2$};	
	\node [draw,circle](3)at(1,3.5){$3$};
	\node [draw,circle](4)at(2.7,2){$4$};
	\node [draw,circle](5)at(2.2,0){$5$};
	\node [draw,circle](6)at(4.2,-1){$6$};
	\node [below] at (-1.7,0.3) { $(60)$};
	\node [below] at (3.7,0.3) {$(60)$};
	\node [below] at (1,0) {$(6)$};
	\node [anchor=south east]  at (-0.2,3.4) { $(6)$};
	\node [above]  at (-1.4,2.7) {$(6)$};
	\node [right] at (2.4,0.9) {$(21)$};
	\node [above] at (3.3,3.3){$(6)$};
	\node [left] at (-0.4,0.9) { $(6)$};
	\draw (0) to (1);
	\draw (1) to (2);
	\draw (1) to (3);
	\draw (1) to (4);
	\draw (1) to (5);
	
	\draw (2) to (3);   	
	\draw (2) to (4);
	\draw (2) to (5);
	
	\draw (3) to (4);
	\draw (3) to (5);
	
	\draw (4) to (5);
	
	\draw (5) to (6);

	\end{tikzpicture}
%
%
%
%
%
%
%
%
\end{minipage}	
\caption{Examples of total weak harmonic labelings obtained from Algorithm \ref{algovalues}. The label of the edges appear in parenthesis (labels equal to $1$ are omitted).}
\label{Figure:MultilabelingTWO}
\end{figure}
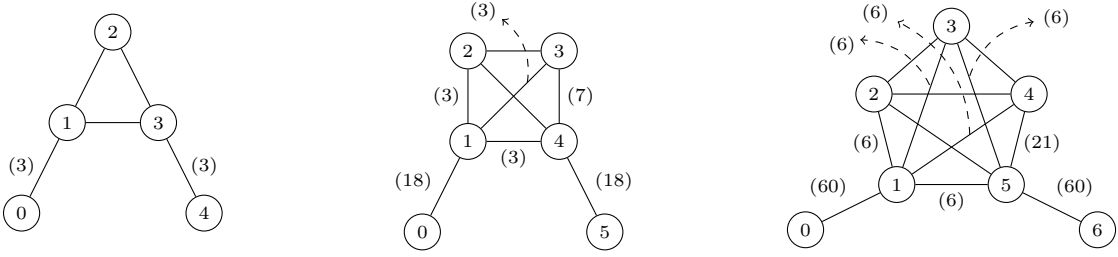

%
%
%

\newpage

\section*{Appendix}

All possible weakly labeled finite graphs up to ten vertices.
 
\begin{table}[h]

\tiny

\vspace{5mm}

\begin{minipage}{0.49\textwidth}
\begin{center}
	\begin{tikzpicture}[scale=0.8]

	\node [draw,circle](0)at(0,0){$0$};
	\node [draw,circle](1)at(2,0){$1$};
	\node [draw,circle](2)at(4,0){$2$};

	\draw (1)--(0);
	\draw (1)--(2);

	\end{tikzpicture}\\
	\vspace{5mm}
	$ \mathcal{A}=\lbrace 012\rbrace$	
	\end{center}
\end{minipage}
\;
\begin{minipage}{0.49\textwidth}
\begin{center}
	\begin{tikzpicture}[scale=0.8]

	\node [draw,circle](0)at(0,0){$0$};
	\node [draw,circle](1)at(2,0){$1$};
	\node [draw,circle](2)at(4,0){$2$};
	\node [draw,circle](3)at(6,0){$3$};
	
	\draw (1)--(0);
	\draw (1)--(2);
	\draw (3)--(2);
	\end{tikzpicture}\\
	\vspace{5mm}
	$ \mathcal{A}=\lbrace 012,123\rbrace$
	
	\end{center}
\end{minipage}\\

\vspace{5mm}

\begin{minipage}{0.49\textwidth}
\begin{center}
	\begin{tikzpicture}[scale=0.6]

	\node [draw,circle](0)at(0,0){$0$};
	\node [draw,circle](1)at(2,0){$1$};
	\node [draw,circle](2)at(4,0){$2$};
	\node [draw,circle](3)at(6,0){$3$};
	\node [draw,circle](4)at(8,0){$4$};
	\draw (1)--(0);
	\draw (1)--(2);
	\draw (3)--(2);
	\draw (3)--(4);
	\end{tikzpicture}\\
	\vspace{5mm}
	$ \mathcal{A}=\lbrace 012; 123; 234\rbrace$	
\end{center}
\end{minipage}
\;
\begin{minipage}{0.49\textwidth}
\begin{center}
	\begin{tikzpicture}[scale=0.6]

	\node [draw,circle](2)at(0,0){$2$};
	
	\node [draw,circle](1)at(2,0){$1$};
	\node [draw,circle](0)at(-2,0){$0$};
	\node [draw,circle](3)at(0,2){$3$};
	\node [draw,circle](4)at(0,-2){$4$};
	\draw (1)--(2);
	\draw (0)--(2);
	\draw (3)--(2);
	\draw (4)--(2);
	\end{tikzpicture}\\
	\vspace{5mm}
	$ \mathcal{A}=\lbrace 01234\rbrace$	
\end{center}
\end{minipage}

\vspace{5mm}

\begin{center}

	\begin{tikzpicture}[scale=0.6]

	\node [draw,circle](0)at(0,0){$0$};
	\node [draw,circle](1)at(2,0){$1$};
	\node [draw,circle](2)at(4,0){$2$};
	\node [draw,circle](3)at(6,0){$3$};
	\node [draw,circle](4)at(8,0){$4$};
	\node [draw,circle](5)at(10,0){$5$};
	\draw (1)--(0);
	\draw (1)--(2);
	\draw (3)--(2);
	\draw (3)--(4);
	\draw (5)--(4);
	\end{tikzpicture}\\
	\vspace{5mm}
	$ \mathcal{A}=\lbrace 012; 123; 234; 345\rbrace$	
\end{center}

\vspace{5mm}

\begin{minipage}{0.49\textwidth}
\begin{center}
	\begin{tikzpicture}[scale=0.55]

	\node [draw,circle](0)at(0,0){$0$};
	\node [draw,circle](1)at(2,0){$1$};
	\node [draw,circle](2)at(4,0){$2$};
	\node [draw,circle](3)at(6,0){$3$};
	\node [draw,circle](4)at(8,0){$4$};
	\node [draw,circle](5)at(10,0){$5$};
	\node [draw,circle](6)at(12,0){$6$};
	
	\draw (1)--(0);
	\draw (1)--(2);
	\draw  (2) to (3);
	\draw  (3) to (4);
	\draw  (4) to (5);
	\draw  (5) to (6);
	
	\end{tikzpicture}\\
	\vspace{5mm}
	$ \mathcal{A}=\lbrace 012; 123; 234;345, 456\rbrace$
	\end{center}
\end{minipage}
\;
\begin{minipage}{0.49\textwidth}
\begin{center}
	\begin{tikzpicture}[scale=0.6]
	
	\node [draw,circle](1)at(0,0){$1$};
	\node [draw,circle](2) at(2,0){$2$};
	\node [draw,circle](3)at(4,0){$3$};
	\node [draw,circle](4)at(6,0){$4$};
	\node [draw,circle](5)at(8,0){$5$};
	\node [draw,circle](0)at(4,2){$0$};
	\node [draw,circle](6)at(4,-2){$6$};

	%
	
	
	\draw (1) to (2);
	\draw (2)--(3);
	\draw (3)--(4);
	\draw (4)-- (5);
	\draw (0)-- (3);
	\draw (3)--(6);

	\end{tikzpicture}\\
	\vspace{5mm}
$ \mathcal{A}=\lbrace 123;02346;345\rbrace$		
\end{center}
\end{minipage}\\
\begin{minipage}{0.49\textwidth}
\begin{center}
	\begin{tikzpicture}[scale=0.6]

	\node [draw,circle](3)at(0,0){$3$};
	
	\node [draw,circle](0)at(-3.8,0){$0$};
	\node [draw,circle](5)at(3.8,0){$5$};
	\node [draw,circle](6)at(2,-3){$6$};
	\node [draw,circle](1)at(-2,-3){$1$};
	\node [draw,circle](2)at(-2,3){$2$};
	\node [draw,circle](4)at(2,3){$4$};

	%
	
	\draw  (3) to (0);
	\draw  (3) to (1);
	\draw  (3) to (2);
	\draw  (3) to (4);
	\draw  (3) to (5);
	\draw  (3) to (6);
	\end{tikzpicture}\\
	\vspace{5mm}
	$ \mathcal{A}=\lbrace 0123456\rbrace$	
\end{center}
\end{minipage}
\;
\begin{minipage}{0.49\textwidth}
\begin{center}
 
	\begin{tikzpicture}[scale=0.6]
	
	\node [draw,circle](0)at(-2,-2){$0$};
	\node [draw,circle](1) at(-2,2){$1$};
	\node [draw,circle](2)at(0,0){$2$};
	\node [draw,circle](3)at(2,-2){$3$};
	\node [draw,circle](4)at(4,0){$4$};
	\node [draw,circle](5)at(6,2){$5$};
	\node [draw,circle](6)at(6,-2){$6$};

	%
	
	
	\draw (2) to (0);
	\draw (2)--(1);
	\draw (2)--(3);
	\draw (2)-- (4);
	\draw (3)-- (4);
	\draw (4)--(5);
	\draw (4)--(6);

	\end{tikzpicture}\\
	\vspace{5mm}	
	$ \mathcal{A}=\lbrace 01234; 234; 23456\rbrace$
\end{center}

\end{minipage}

\vspace{5mm}

\caption{Every possible weakly labeled graph of up to seven vertices.}
\end{table}


\newpage

\begin{table}

\tiny

\vspace{5mm}

\begin{minipage}{0.49\textwidth}
\begin{center}
	\begin{tikzpicture}[scale=0.45]

	\node [draw,circle](0)at(0,0){$0$};
	\node [draw,circle](1)at(2,0){$1$};
	\node [draw,circle](2)at(4,0){$2$};
	\node [draw,circle](3)at(6,0){$3$};
	\node [draw,circle](4)at(8,0){$4$};
	\node [draw,circle](5)at(10,0){$5$};
	\node [draw,circle](6)at(12,0){$6$};
	\node [draw,circle](7)at(14,0){$7$};
	
	\draw (1)--(0);
	\draw (1)--(2);
	\draw  (2) to (3);
	\draw  (3) to (4);
	\draw  (4) to (5);
	\draw  (5) to (6);
	\draw  (7) to (6);
	
	\end{tikzpicture}\\
	\vspace{2cm}
	$ \mathcal{A}=\lbrace 012; 123; 234;345;456;567\rbrace$	
\end{center}
\end{minipage}
\;
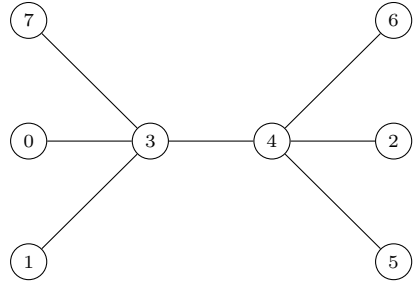
\begin{minipage}{0.49\textwidth}
\begin{center}
	\begin{tikzpicture}[scale=0.8]
	
	\node [draw,circle](0)at(0,0){$0$};
	\node [draw,circle](4) at(2,0){$3$};
	\node [draw,circle](3)at(4,0){$4$};
	\node [draw,circle](1)at(6,0){$2$};
	\node [draw,circle](6)at(0,2){$7$};
	\node [draw,circle](7)at(0,-2){$1$};
	\node [draw,circle](2)at(6,2){$6$};
	\node [draw,circle](5)at(6,-2){$5$};

	\draw (0) to (4);
	\draw (3) to (4);
	\draw (6) to (4);
	\draw (7) to (4);
	
	\draw (2)--(3);
	\draw (1)--(3);
	\draw (5)--(3);

	\end{tikzpicture}\\
	\vspace{5mm}	
	$ \mathcal{A}=\lbrace 01347; 13457\rbrace$
\end{center}
\end{minipage}\\

\vspace{10mm}

\begin{minipage}{0.49\textwidth}
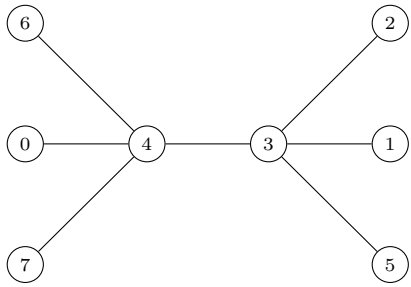
\begin{center}
	\begin{tikzpicture}[scale=0.8]
	
	\node [draw,circle](0)at(0,0){$0$};
	\node [draw,circle](4) at(2,0){$4$};
	\node [draw,circle](3)at(4,0){$3$};
	\node [draw,circle](1)at(6,0){$1$};
	\node [draw,circle](6)at(0,2){$6$};
	\node [draw,circle](7)at(0,-2){$7$};
	\node [draw,circle](2)at(6,2){$2$};
	\node [draw,circle](5)at(6,-2){$5$};

	\draw (0) to (4);
	\draw (3) to (4);
	\draw (6) to (4);
	\draw (7) to (4);
	
	\draw (2)--(3);
	\draw (1)--(3);
	\draw (5)--(3);

	\end{tikzpicture}\\
	\vspace{5mm}
	$ \mathcal{A}=\lbrace 03467; 12345\rbrace$
\end{center}
\end{minipage}
\;
\begin{minipage}{0.49\textwidth}
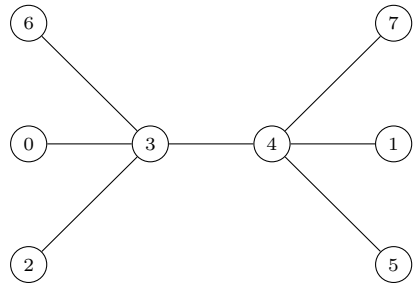
\begin{center}
	\begin{tikzpicture}[scale=0.8]
	
	\node [draw,circle](0)at(0,0){$0$};
	\node [draw,circle](4) at(2,0){$3$};
	\node [draw,circle](3)at(4,0){$4$};
	\node [draw,circle](1)at(6,0){$1$};
	\node [draw,circle](6)at(0,2){$6$};
	\node [draw,circle](7)at(0,-2){$2$};
	\node [draw,circle](2)at(6,2){$7$};
	\node [draw,circle](5)at(6,-2){$5$};

	\draw (0) to (4);
	\draw (3) to (4);
	\draw (6) to (4);
	\draw (7) to (4);
	
	\draw (2)--(3);
	\draw (1)--(3);
	\draw (5)--(3);

	\end{tikzpicture}\\
	\vspace{5mm}
	$ \mathcal{A}=\lbrace 02346; 13457\rbrace$	
\end{center}
\end{minipage}\

\vspace{10mm}

\begin{minipage}{0.49\textwidth}
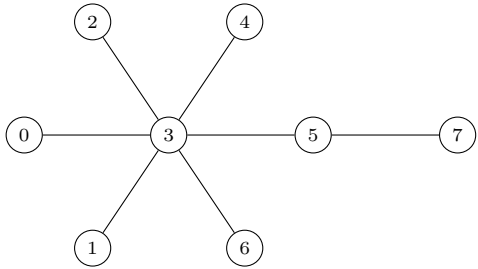
\begin{center}
	\begin{tikzpicture}[scale=0.5]

	\node [draw,circle](3)at(0,0){$3$};
	
	\node [draw,circle](0)at(-3.8,0){$0$};
	\node [draw,circle](5)at(3.8,0){$5$};
	\node [draw,circle](6)at(2,-3){$6$};
	\node [draw,circle](1)at(-2,-3){$1$};
	\node [draw,circle](2)at(-2,3){$2$};
	\node [draw,circle](4)at(2,3){$4$};
	\node [draw,circle](7)at(7.6,0){$7$};
	
	%
	
	\draw  (3) to (0);
	\draw  (3) to (1);
	\draw  (3) to (2);
	\draw  (3) to (4);
	\draw  (3) to (5);
	\draw  (3) to (6);
	\draw  (5) to (7);
	\end{tikzpicture}\\
	\vspace{5mm}
	$ \mathcal{A}=\lbrace 0123456; 357\rbrace$	
\end{center}
\end{minipage}
\;
\begin{minipage}{0.49\textwidth}
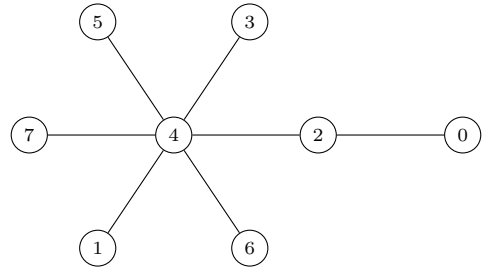
\begin{center}

 \begin{tikzpicture}[scale=0.5]

	\node [draw,circle](3)at(0,0){$4$};
	
	\node [draw,circle](0)at(-3.8,0){$7$};
	\node [draw,circle](5)at(3.8,0){$2$};
	\node [draw,circle](6)at(2,-3){$6$};
	\node [draw,circle](1)at(-2,-3){$1$};
	\node [draw,circle](2)at(-2,3){$5$};
	\node [draw,circle](4)at(2,3){$3$};
	\node [draw,circle](7)at(7.6,0){$0$};
	
	%
	
	\draw  (3) to (0);
	\draw  (3) to (1);
	\draw  (3) to (2);
	\draw  (3) to (4);
	\draw  (3) to (5);
	\draw  (3) to (6);
	\draw  (5) to (7);
	\end{tikzpicture}\\
	\vspace{5mm}
	$ \mathcal{A}=\lbrace 1234567; 024\rbrace$	
\end{center}

\end{minipage}

\vspace{5mm}

\caption{Every possible weakly labeled graph of eight vertices.}
\end{table}


\begin{table}

\tiny

\vspace{5mm}

\begin{minipage}{0.49\textwidth}
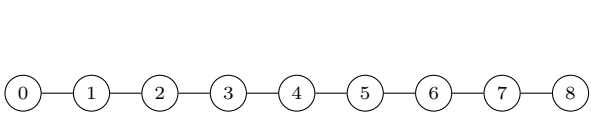
\begin{center}
	\begin{tikzpicture}[scale=0.45]

	\node [draw,circle](0)at(0,0){$0$};
	\node [draw,circle](1)at(2,0){$1$};
	\node [draw,circle](2)at(4,0){$2$};
	\node [draw,circle](3)at(6,0){$3$};
	\node [draw,circle](4)at(8,0){$4$};
	\node [draw,circle](5)at(10,0){$5$};
	\node [draw,circle](6)at(12,0){$6$};
	\node [draw,circle](7)at(14,0){$7$};
	\node [draw,circle](8)at(16,0){$8$};
	\draw (1)--(0);
	\draw (1)--(2);
	\draw  (2) to (3);
	\draw  (3) to (4);
	\draw  (4) to (5);
	\draw  (5) to (6);
	\draw  (6) to (7);
	\draw  (7) to (8);
	\end{tikzpicture}\\
	\vspace{1cm}
	$ \mathcal{A}=\lbrace 012; 123; 234;345;456;567;678\rbrace$	
\end{center}
\end{minipage}
\;
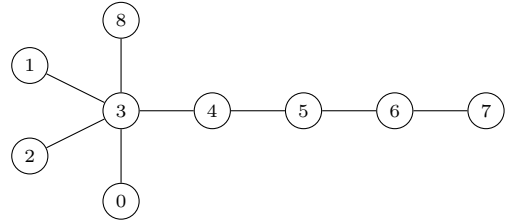
\begin{minipage}{0.49\textwidth}
\begin{center}
	\begin{tikzpicture}[scale=0.6]
	
	\node [draw,circle](3)at(0,0){$3$};
	\node [draw,circle](8) at(0,2){$8$};
	\node [draw,circle](0)at(0,-2){$0$};
	\node [draw,circle](1)at(-2,1){$1$};
	\node [draw,circle](2)at(-2,-1){$2$};
	\node [draw,circle](4)at(2,0){$4$};
	\node [draw,circle](5)at(4,0){$5$};
	\node [draw,circle](6)at(6,0){$6$};
	\node [draw,circle](7)at(8,0){$7$};

	\draw (8) to (3);
	\draw (0) to (3);
	\draw (2) to (3);
	\draw (1) to (3);
	
	\draw (3)--(4);
	\draw (4)--(5);
	\draw (5)--(6);
	\draw (6) to (7);

	\end{tikzpicture}\\
	\vspace{5mm}	
	$ \mathcal{A}=\lbrace 012348; 345; 456; 567\rbrace$
\end{center}
\end{minipage}\

\vspace{10mm}

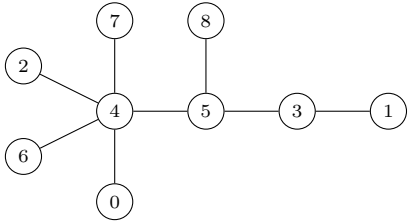
\begin{minipage}{0.49\textwidth}
\begin{center}
	\begin{tikzpicture}[scale=0.6]
	
	\node [draw,circle](4)at(0,0){$4$};
	\node [draw,circle](7) at(0,2){$7$};
	\node [draw,circle](0)at(0,-2){$0$};
	\node [draw,circle](2)at(-2,1){$2$};
	\node [draw,circle](6)at(-2,-1){$6$};
	\node [draw,circle](5)at(2,0){$5$};
	\node [draw,circle](8)at(2,2){$8$};
	\node [draw,circle](3)at(4,0){$3$};
	\node [draw,circle](1)at(6,0){$1$};

	\draw (7) to (4);
	\draw (0) to (4);
	\draw (2) to (4);
	\draw (6) to (4);
	\draw (5) to (4);
	
	\draw (5)--(8);
	\draw (3)--(5);
	\draw (3)--(1);

	\end{tikzpicture}\\
	\vspace{5mm}	
	$ \mathcal{A}=\lbrace 024567; 3458; 135\rbrace$
\end{center}
\end{minipage}
\;
\begin{minipage}{0.49\textwidth}
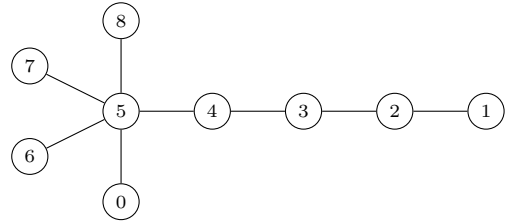
\begin{center}
	\begin{tikzpicture}[scale=0.6]
	
	\node [draw,circle](5)at(0,0){$5$};
	\node [draw,circle](8) at(0,2){$8$};
	\node [draw,circle](0)at(0,-2){$0$};
	\node [draw,circle](7)at(-2,1){$7$};
	\node [draw,circle](6)at(-2,-1){$6$};
	\node [draw,circle](4)at(2,0){$4$};
	\node [draw,circle](3)at(4,0){$3$};
	\node [draw,circle](2)at(6,0){$2$};
	\node [draw,circle](1)at(8,0){$1$};

	\draw (8) to (5);
	\draw (0) to (5);
	\draw (4) to (5);
	\draw (6) to (5);
	\draw (7) to (5);
	\draw (3)--(4);
	\draw (2)--(3);
	\draw (2)--(1);

	\end{tikzpicture}\\
	\vspace{5mm}	
	$ \mathcal{A}=\lbrace 045678; 345; 456; 123\rbrace$
\end{center}

\end{minipage}

\vspace{10mm}

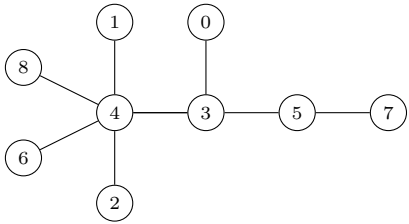
\begin{minipage}{0.49\textwidth}
\begin{center}
	\begin{tikzpicture}[scale=0.6]
	
	\node [draw,circle](4)at(0,0){$4$};
	\node [draw,circle](1) at(0,2){$1$};
	\node [draw,circle](2)at(0,-2){$2$};
	\node [draw,circle](8)at(-2,1){$8$};
	\node [draw,circle](6)at(-2,-1){$6$};
	\node [draw,circle](3)at(2,0){$3$};
	\node [draw,circle](0)at(2,2){$0$};
	\node [draw,circle](5)at(4,0){$5$};
	\node [draw,circle](7)at(6,0){$7$};

	\draw (1) to (4);
	\draw (3) to (4);
	\draw (2) to (4);
	\draw (6) to (4);
	\draw (8) to (4);
	
	\draw (4)--(3);
	\draw (3)--(0);
	\draw (3)--(5);
	\draw (7)--(5);

	\end{tikzpicture}\\
	\vspace{5mm}	
	$ \mathcal{A}=\lbrace 123468; 0345; 357\rbrace$
\end{center}
\end{minipage}
\;
\begin{minipage}{0.49\textwidth}
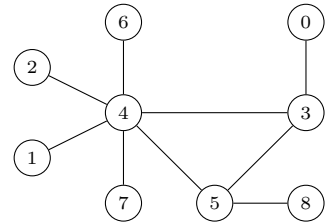
\begin{center}
	\begin{tikzpicture}[scale=0.6]
	
	\node [draw,circle](4)at(0,0){$4$};
	\node [draw,circle](6) at(0,2){$6$};
	\node [draw,circle](7)at(0,-2){$7$};
	\node [draw,circle](2)at(-2,1){$2$};
	\node [draw,circle](1)at(-2,-1){$1$};
	\node [draw,circle](5)at(2,-2){$5$};
	\node [draw,circle](3)at(4,0){$3$};
	\node [draw,circle](0)at(4,2){$0$};
	\node [draw,circle](8)at(4,-2){$8$};

	\draw (7) to (4);
	\draw (2) to (4);
	\draw (1) to (4);
	\draw (6) to (4);
	\draw (3) to (4);
	\draw (5)--(4);
	\draw (5)--(3);
	\draw (0)--(3);
	
	\draw (8)--(5);

	\end{tikzpicture}\\
	\vspace{5mm}	
	$ \mathcal{A}=\lbrace 1234567; 0345; 3458\rbrace$
\end{center}

\end{minipage}

\vspace{10mm}

\begin{minipage}{0.49\textwidth}
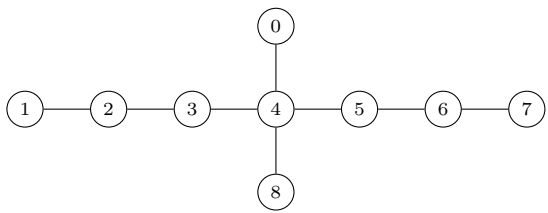
\begin{center}
	\begin{tikzpicture}[scale=0.55]

	\node [draw,circle](0)at(8,2){$0$};
	\node [draw,circle](1)at(2,0){$1$};
	\node [draw,circle](2)at(4,0){$2$};
	\node [draw,circle](3)at(6,0){$3$};
	\node [draw,circle](4)at(8,0){$4$};
	\node [draw,circle](5)at(10,0){$5$};
	\node [draw,circle](6)at(12,0){$6$};
	\node [draw,circle](7)at(14,0){$7$};
	\node [draw,circle](8)at(8,-2){$8$};

	\draw (4)--(0);
	\draw (1)--(2);
	\draw  (2) to (3);
	\draw  (3) to (4);
	\draw  (4) to (5);
	\draw  (4) to (8);
	\draw  (5) to (6);
	\draw  (7) to (6);
	
	\end{tikzpicture}\\
	\vspace{5mm}
	$ \mathcal{A}=\lbrace  123; 234;03458;456;567\rbrace$	
\end{center}
\end{minipage}
\;
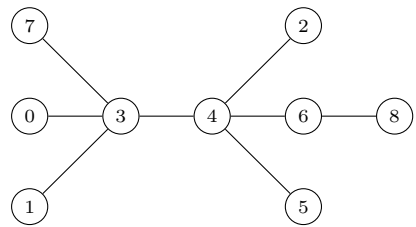
\begin{minipage}{0.49\textwidth}
\begin{center}
	\begin{tikzpicture}[scale=0.6]
	
	\node [draw,circle](0)at(0,0){$0$};
	\node [draw,circle](4) at(2,0){$3$};
	\node [draw,circle](3)at(4,0){$4$};
	\node [draw,circle](1)at(6,0){$6$};
	\node [draw,circle](6)at(0,2){$7$};
	\node [draw,circle](7)at(0,-2){$1$};
	\node [draw,circle](2)at(6,2){$2$};
	\node [draw,circle](5)at(6,-2){$5$};
	\node [draw,circle](8)at(8,0){$8$};

	\draw (0) to (4);
	\draw (3) to (4);
	\draw (6) to (4);
	\draw (7) to (4);
	
	\draw (2)--(3);
	\draw (1)--(3);
	\draw (5)--(3);
	\draw (1)--(8);

	\end{tikzpicture}\\
	\vspace{5mm}	
	$ \mathcal{A}=\lbrace 01347; 23456; 468\rbrace$
\end{center}

\end{minipage}

\vspace{10mm}

\begin{minipage}{0.49\textwidth}
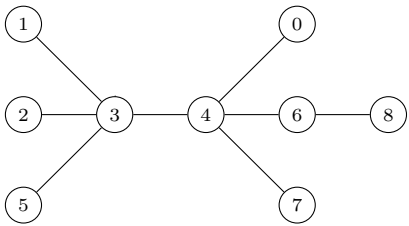
\begin{center}
	\begin{tikzpicture}[scale=0.6]
	
	\node [draw,circle](2)at(-2,0){$2$};
	\node [draw,circle](1) at(-2,2){$1$};
	\node [draw,circle](5)at(-2,-2){$5$};
	\node [draw,circle](3)at(0,0){$3$};
	\node [draw,circle](4)at(2,0){$4$};
	\node [draw,circle](7)at(4,-2){$7$};
	\node [draw,circle](6)at(4,0){$6$};
	\node [draw,circle](0)at(4,2){$0$};
	\node [draw,circle](8)at(6,0){$8$};

	\draw (1) to (3);
	\draw (2) to (3);
	\draw (5) to (3);
	\draw (3) to (3);
	\draw (4)--(3);
	
	\draw (0)--(4);
	\draw (7)--(4);
	\draw (6)--(4);
	\draw (6)--(8);
	
	\end{tikzpicture}\\
	\vspace{5mm}	
	$ \mathcal{A}=\lbrace 12345; 03467; 468\rbrace$
\end{center}
\end{minipage}
\;
\begin{minipage}{0.49\textwidth}
\begin{center}
	\begin{tikzpicture}[scale=0.6]
	
	\node [draw,circle](8)at(-2,0){$8$};
	\node [draw,circle](1) at(-2,2){$1$};
	\node [draw,circle](7)at(-2,-2){$7$};
	\node [draw,circle](5)at(0,0){$5$};
	\node [draw,circle](4)at(2,0){$4$};
	\node [draw,circle](3)at(4,0){$3$};
	\node [draw,circle](6)at(6,2){$6$};
	\node [draw,circle](0)at(6,0){$0$};
	\node [draw,circle](2)at(6,-2){$2$};

	\draw (1) to (5);
	\draw (4) to (5);
	\draw (8) to (5);
	\draw (7) to (5);
	
	\draw (4)--(3);
	
	\draw (0)--(3);
	\draw (2)--(3);
	\draw (6)--(3);

	\end{tikzpicture}\\
	\vspace{5mm}	
	$ \mathcal{A}=\lbrace 14578; 345; 02346\rbrace$
\end{center}

\end{minipage}

\vspace{5mm}

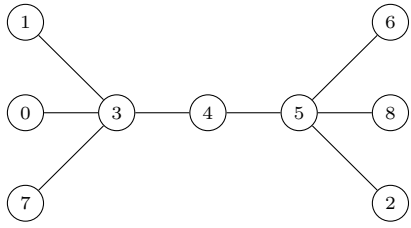
\caption{Every possible weakly labeled graph of nine vertices (table  1 of 3).}
\end{table}

\begin{table}

\tiny

\vspace{5mm}

\begin{minipage}{0.49\textwidth}
\begin{center}
	\begin{tikzpicture}[scale=0.6]
	
	\node [draw,circle](8)at(-2,0){$0$};
	\node [draw,circle](1) at(-2,2){$1$};
	\node [draw,circle](7)at(-2,-2){$7$};
	\node [draw,circle](5)at(0,0){$3$};
	\node [draw,circle](4)at(2,0){$4$};
	\node [draw,circle](3)at(4,0){$5$};
	\node [draw,circle](6)at(6,2){$6$};
	\node [draw,circle](0)at(6,0){$8$};
	\node [draw,circle](2)at(6,-2){$2$};

	\draw (1) to (5);
	\draw (4) to (5);
	\draw (8) to (5);
	\draw (7) to (5);
	
	\draw (4)--(3);
	
	\draw (0)--(3);
	\draw (2)--(3);
	\draw (6)--(3);

	\end{tikzpicture}\\
	\vspace{5mm}	
	$ \mathcal{A}=\lbrace 01347; 345; 2456\rbrace$
\end{center}
\end{minipage}
\;
\begin{minipage}{0.49\textwidth}
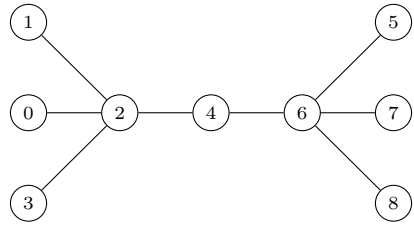
\begin{center}
	\begin{tikzpicture}[scale=0.6]
	
	\node [draw,circle](8)at(-2,0){$0$};
	\node [draw,circle](1) at(-2,2){$1$};
	\node [draw,circle](7)at(-2,-2){$3$};
	\node [draw,circle](5)at(0,0){$2$};
	\node [draw,circle](4)at(2,0){$4$};
	\node [draw,circle](3)at(4,0){$6$};
	\node [draw,circle](6)at(6,2){$5$};
	\node [draw,circle](0)at(6,0){$7$};
	\node [draw,circle](2)at(6,-2){$8$};

	\draw (1) to (5);
	\draw (4) to (5);
	\draw (8) to (5);
	\draw (7) to (5);
	
	\draw (4)--(3);
	
	\draw (0)--(3);
	\draw (2)--(3);
	\draw (6)--(3);

	\end{tikzpicture}\\
	\vspace{5mm}	
	$ \mathcal{A}=\lbrace 01234; 246; 45678\rbrace$
\end{center}
\end{minipage}\

\vspace{10mm}

\begin{minipage}{0.49\textwidth}
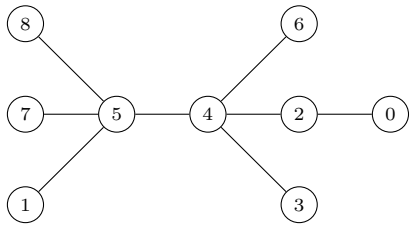
\begin{center}
	\begin{tikzpicture}[scale=0.6]
	
	\node [draw,circle](7)at(-2,0){$7$};
	\node [draw,circle](8) at(-2,2){$8$};
	\node [draw,circle](1)at(-2,-2){$1$};
	\node [draw,circle](5)at(0,0){$5$};
	\node [draw,circle](4)at(2,0){$4$};
	\node [draw,circle](2)at(4,0){$2$};
	\node [draw,circle](6)at(4,2){$6$};
	\node [draw,circle](3)at(4,-2){$3$};
	\node [draw,circle](0)at(6,0){$0$};

	\draw (1) to (5);
	\draw (4) to (5);
	\draw (8) to (5);
	\draw (7) to (5);
	
	\draw (4)--(3);
	\draw (4)--(2);
	\draw (4)--(6);
	
	\draw (0)--(2);

	\end{tikzpicture}\\
	\vspace{5mm}	
	$ \mathcal{A}=\lbrace 14578; 23456; 024\rbrace$
\end{center}
\end{minipage}
\;
\begin{minipage}{0.49\textwidth}
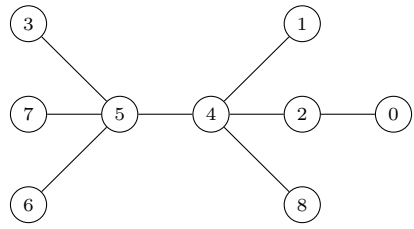
\begin{center}
	\begin{tikzpicture}[scale=0.6]
	
	\node [draw,circle](7)at(-2,0){$7$};
	\node [draw,circle](8) at(-2,2){$3$};
	\node [draw,circle](1)at(-2,-2){$6$};
	\node [draw,circle](5)at(0,0){$5$};
	\node [draw,circle](4)at(2,0){$4$};
	\node [draw,circle](2)at(4,0){$2$};
	\node [draw,circle](6)at(4,2){$1$};
	\node [draw,circle](3)at(4,-2){$8$};
	\node [draw,circle](0)at(6,0){$0$};

	\draw (1) to (5);
	\draw (4) to (5);
	\draw (8) to (5);
	\draw (7) to (5);
	
	\draw (4)--(3);
	\draw (4)--(2);
	\draw (4)--(6);
	
	\draw (0)--(2);

	\end{tikzpicture}\\
	\vspace{5mm}	
	$ \mathcal{A}=\lbrace 34567; 12458; 024\rbrace$
\end{center}

\end{minipage}

\vspace{10mm}

\begin{minipage}{0.49\textwidth}
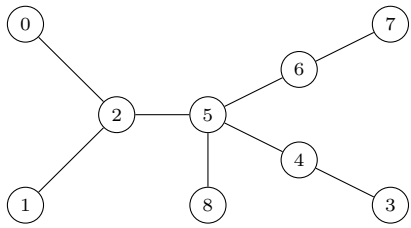
\begin{center}
	\begin{tikzpicture}[scale=0.6]
	
	\node [draw,circle](0)at(-2,2){$0$};
	\node [draw,circle](1) at(-2,-2){$1$};
	\node [draw,circle](2)at(0,0){$2$};
	\node [draw,circle](5)at(2,0){$5$};
	\node [draw,circle](8)at(2,-2){$8$};
	\node [draw,circle](6)at(4,1){$6$};
	\node [draw,circle](4)at(4,-1){$4$};
	\node [draw,circle](3)at(6,-2){$3$};
	\node [draw,circle](7)at(6,2){$7$};

	\draw (0) to (2);
	\draw (1) to (2);
	\draw (5) to (2);

	\draw (8) to (5);
	\draw (6) to (5);
	\draw (4)--(5);
	\draw (4)--(3);
	\draw (7)--(6);

	\end{tikzpicture}\\
	\vspace{5mm}	
	$ \mathcal{A}=\lbrace 0125; 24568; 345; 567\rbrace$
\end{center}
\end{minipage}
\;
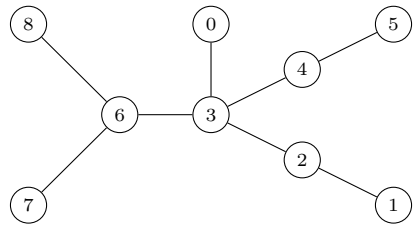
\begin{minipage}{0.49\textwidth}
\begin{center}
	\begin{tikzpicture}[scale=0.6]
	
	\node [draw,circle](7)at(-2,-2){$7$};
	\node [draw,circle](8) at(-2,2){$8$};
	\node [draw,circle](6)at(0,0){$6$};
	\node [draw,circle](3)at(2,0){$3$};
	\node [draw,circle](0)at(2,2){$0$};
	\node [draw,circle](4)at(4,1){$4$};
	\node [draw,circle](5)at(6,2){$5$};
	\node [draw,circle](2)at(4,-1){$2$};
	\node [draw,circle](1)at(6,-2){$1$};

	\draw (3) to (6);
	\draw (8) to (6);
	\draw (7) to (6);
	
	\draw (0)--(3);
	\draw (3)--(2);
	\draw (4)--(3);
	\draw (4)--(5);
	\draw (1)--(2);

	\end{tikzpicture}\\
	\vspace{5mm}	
	$ \mathcal{A}=\lbrace 3678; 02346; 123; 345\rbrace$
\end{center}

\end{minipage}

\vspace{10mm}

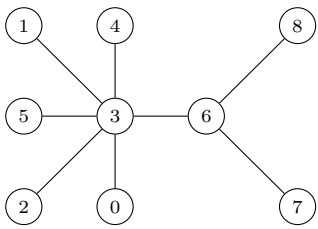
\begin{minipage}{0.49\textwidth}
\begin{center}
	\begin{tikzpicture}[scale=0.6]
	
	\node [draw,circle](5)at(-2,0){$5$};
	\node [draw,circle](1) at(-2,2){$1$};
	\node [draw,circle](2)at(-2,-2){$2$};
	\node [draw,circle](3)at(0,0){$3$};
	\node [draw,circle](4)at(0,2){$4$};
	\node [draw,circle](0)at(0,-2){$0$};
	\node [draw,circle](6)at(2,0){$6$};
	\node [draw,circle](8)at(4,2){$8$};
	\node [draw,circle](7)at(4,-2){$7$};

	\draw (0) to (3);
	\draw (1) to (3);
	\draw (2) to (3);
	\draw (4) to (3);
	\draw (5) to (3);
	\draw (6) to (3);
	
	\draw (8) to (6);
	\draw (6) to (7);

	\end{tikzpicture}\\
	\vspace{5mm}	
	$ \mathcal{A}=\lbrace 0123456; 3678\rbrace$
\end{center}
\end{minipage}
\;
\begin{minipage}{0.49\textwidth}
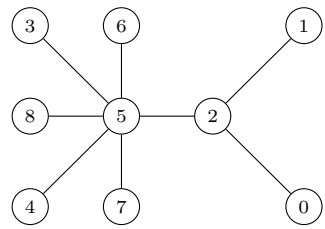
\begin{center}
	\begin{tikzpicture}[scale=0.6]
	
	\node [draw,circle](8)at(-2,0){$8$};
	\node [draw,circle](3) at(-2,2){$3$};
	\node [draw,circle](4)at(-2,-2){$4$};
	\node [draw,circle](6)at(0,2){$6$};
	\node [draw,circle](5)at(0,0){$5$};
	\node [draw,circle](7)at(0,-2){$7$};
	\node [draw,circle](2)at(2,0){$2$};
	\node [draw,circle](1)at(4,2){$1$};
	\node [draw,circle](0)at(4,-2){$0$};

	\draw (2) to (5);
	\draw (3) to (5);
	\draw (4) to (5);
	\draw (6) to (5);
	\draw (7) to (5);
	\draw (8) to (5);
	
	\draw (0)--(2);
	\draw (1)--(2);

	\end{tikzpicture}\\
	\vspace{5mm}	
	$ \mathcal{A}=\lbrace 0125; 2345678\rbrace$
\end{center}

\end{minipage}

\vspace{10mm}

\begin{minipage}{0.49\textwidth}
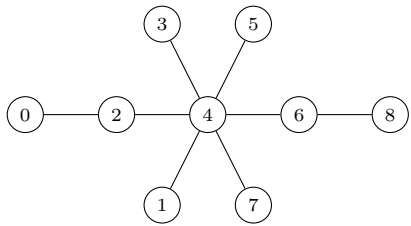
\begin{center}
	\begin{tikzpicture}[scale=0.6]
	
	\node [draw,circle](0)at(0,0){$0$};
	\node [draw,circle](2) at(2,0){$2$};
	\node [draw,circle](4)at(4,0){$4$};
	\node [draw,circle](3)at(3,2){$3$};
	\node [draw,circle](5)at(5,2){$5$};
	\node [draw,circle](1)at(3,-2){$1$};
	\node [draw,circle](7)at(5,-2){$7$};
	\node [draw,circle](6)at(6,0){$6$};
	\node [draw,circle](8)at(8,0){$8$};

	\draw (0) to (2);
	
	\draw (1) to (4);
	\draw (2) to (4);
	\draw (3) to (4);
	\draw (5) to (4);
	\draw (6) to (4);
	\draw (7) to (4);
	
	\draw (8) to (6);

	\end{tikzpicture}\\
	\vspace{5mm}	
	$ \mathcal{A}=\lbrace 024; 1234567; 468\rbrace$
\end{center}
\end{minipage}
\;
\begin{minipage}{0.49\textwidth}
\begin{center}
	\begin{tikzpicture}[scale=0.6]
	
	\node [draw,circle](0)at(0,0){$2$};
	\node [draw,circle](2) at(2,0){$3$};
	\node [draw,circle](4)at(4,0){$4$};
	\node [draw,circle](3)at(3,2){$0$};
	\node [draw,circle](5)at(5,2){$1$};
	\node [draw,circle](1)at(3,-2){$8$};
	\node [draw,circle](7)at(5,-2){$7$};
	\node [draw,circle](6)at(6,0){$5$};
	\node [draw,circle](8)at(8,0){$6$};

	\draw (0) to (2);
	
	\draw (1) to (4);
	\draw (2) to (4);
	\draw (3) to (4);
	\draw (5) to (4);
	\draw (6) to (4);
	\draw (7) to (4);
	
	\draw (8) to (6);

	\end{tikzpicture}\\
	\vspace{5mm}	
	$ \mathcal{A}=\lbrace 234; 0134578; 456\rbrace$
\end{center}

\end{minipage}

\vspace{5mm}

\caption{Every possible weakly labeled graph of nine vertices (table 2 of 3).}
\end{table}

\newpage

\begin{table}[h]

\tiny

\vspace{5mm}

\begin{minipage}{0.49\textwidth}
\begin{center}
	\begin{tikzpicture}[scale=0.6]
	
	\node [draw,circle](4)at(0,0){$4$};
	
	\node [draw,circle](0) at(-2,0){$0$};
	\node [draw,circle](1)at(-2,2){$1$};
	\node [draw,circle](8)at(-2,-2){$8$};
	\node [draw,circle](2)at(0,2){$2$};
	\node [draw,circle](7)at(0,-2){$7$};
	\node [draw,circle](3)at(2,2){$3$};
	\node [draw,circle](6)at(2,-2){$6$};
	\node [draw,circle](5)at(2,0){$5$};

	\draw (0) to (4);
	
	\draw (1) to (4);
	\draw (2) to (4);
	\draw (3) to (4);
	\draw (5) to (4);
	\draw (6) to (4);
	\draw (7) to (4);
	
	\draw (8) to (4);

	\end{tikzpicture}\\
	\vspace{5mm}	
	$ \mathcal{A}=\lbrace 012345678\rbrace$
\end{center}
\end{minipage}
\;
\begin{minipage}{0.49\textwidth}
\begin{center}
	\begin{tikzpicture}[scale=0.6]
	
	\node [draw,circle](6)at(-2,2){$6$};
	\node [draw,circle](5) at(0,0){$5$};
	\node [draw,circle](4)at(4,0){$4$};
	\node [draw,circle](0)at(6,2){$0$};
	
	\node [draw,circle](7)at(-2,-2){$7$};
	\node [draw,circle](8)at(6,-2){$8$};
	
	\node [draw,circle](3)at(2,2){$3$};
	\node [draw,circle](1)at(0,4){$1$};
	\node [draw,circle](2)at(4,4){$2$};


	\draw (1) to (3);
	\draw (2) to (3);
	\draw (4) to (3);
	\draw (5) to (3);
	
	\draw (4) to (5);
	\draw (6) to (5);
	\draw (7) to (5);
	
	\draw (8) to (4);
	\draw (0) to (4);

	\end{tikzpicture}\\
	\vspace{5mm}	
	$ \mathcal{A}=\lbrace 12345; 03458; 34567\rbrace$
\end{center}
\end{minipage}\\

\vspace{5mm}

\begin{center}
	\begin{tikzpicture}[scale=0.6]
	
	\node [draw,circle](7)at(-2,2){$7$};
	\node [draw,circle](8) at(-2,-2){$8$};
	\node [draw,circle](6)at(0,0){$6$};
	\node [draw,circle](5)at(2,-2){$5$};
	\node [draw,circle](4)at(4,0){$4$};
	\node [draw,circle](3)at(6,2){$3$};
	\node [draw,circle](2)at(8,0){$2$};
	\node [draw,circle](0)at(10,2){$0$};
	\node [draw,circle](1)at(10,-2){$1$};

	\draw (4) to (6);
	\draw (5) to (6);
	\draw (7) to (6);
	\draw (8) to (6);

	\draw (2) to (4);
	\draw (3) to (4);
	\draw (5) to (4);
	
	\draw (0) to (2);
	\draw (1) to (2);
	
	\draw (3) to (2);

	\end{tikzpicture}\\
	\vspace{5mm}	
	$ \mathcal{A}=\lbrace 01234; 234; 23456; 456; 45678 \rbrace$
\end{center}

\vspace{5mm}

\caption{Every possible weakly labeled graph of nine vertices (table 3 of 3).}
\end{table}

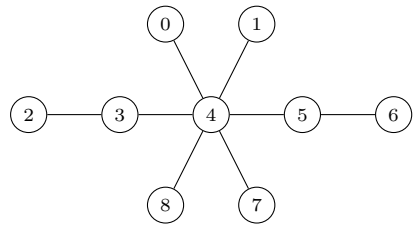
\end{document}